\documentclass[10pt]{article}
\usepackage{enumitem}
\usepackage{amsfonts}
\usepackage{amsmath}
\usepackage{graphicx}
\usepackage{epstopdf}
\usepackage{fancyhdr}
\usepackage{amsthm}
\usepackage{geometry}
\usepackage{subfigure}
\usepackage[figuresright]{rotating}
\usepackage{caption}
\usepackage{bm}
\usepackage{amssymb}
\usepackage{amscd}
\usepackage{float}
\usepackage{color}
\usepackage{tabularx}
\usepackage{diagbox}
\usepackage{tabularx}
\usepackage{mathtools}
\usepackage{multicol}
\usepackage{multirow}
\usepackage{makecell}
\usepackage[justification=centering]{caption}
\usepackage{rotfloat}
\usepackage{longtable}
\allowdisplaybreaks[4]

\newtheorem{theorem}{Theorem}[section]

\newtheorem{remark}{Remark}[section]
\newtheorem{proposition}{Proposition}[section]
\newtheorem{definition}{Definition}[section]

\newcounter{nextauthor}
\def\mathrm{\mbox}

\numberwithin{remark}{section}

\begin{document}
\title{{\bf Consumption and portfolio optimization solvable problems with recursive preferences}\thanks{This work was supported by the National Natural Science Foundation of China (12171339), the Scientific and Technological Research Program of Chongqing Municipal Education Commission (KJQN202400819), the grant from Chongqing Technology and Business University (2356004) and the Fundamental Research Funds for the Central Universities (2682023CX071).}}
\author{Jian-hao Kang$^a$, Zhun Gou$^b$ and Nan-jing Huang$^c$ \thanks{Corresponding author: nanjinghuang@hotmail.com; njhuang@scu.edu.cn}\\
{\small a. School of Mathematics, Southwest Jiaotong University, Chengdu, Sichuan 610031, P.R. China}\\
{\small b. College of Mathematics and Statistics, Chongqing Technology and Business University,}\\
{\small Chongqing 400067, P.R. China}\\
{\small c. Department of Mathematics, Sichuan University, Chengdu, Sichuan 610064, P.R. China} }
\date{}
\maketitle \vspace*{-9mm}
\begin{abstract}
\noindent This paper considers the consumption and portfolio optimization problems with recursive preferences in both infinite and finite time regions, in which the financial market consists of a risk-free asset and a risky asset  following a general stochastic volatility process. By using Bellman's dynamic programming principle, the Hamilton-Jacobi-Bellman (HJB) equation is derived for characterizing the optimal consumption-investment strategy and the corresponding value function. Based on the conjecture of the exponential-polynomial form of the value function under mild conditions, we prove that, when the order of the polynomial $n\leq2$,  the HJB equation has an analytical solution if the investor with unit elasticity of intertemporal substitution and an approximate solution by the log-linear approximation method otherwise. We also prove that the HJB equation has no solutions under the conjecture of the exponential-polynomial form of the value function when the order of the polynomial $n>2$. Finally, the optimal consumption-portfolio strategies to Heston's model are provided and some numerical experiments are given to illustrate the behavior of the optimal consumption-portfolio strategies.
\\ \ \\
\noindent {\bf Keywords}: Stochastic volatility; consumption and investment; recursive preferences; HJB equation; Heston's model.
\\ \ \\
\noindent \textbf{AMS Subject Classification:}  93E20, 91G10, 91G80, 60H30.
\end{abstract}

\section{Introduction}
Since the seminal work by Merton \cite{Merton1969, Merton1971}, the investment and consumption problems have been studied extensively under the expected utility maximization criterion, in which the assumption of constant volatility for risky assets does not correspond with real situations. Numerous empirical studies documented that the stochastic volatility model can be used to capture actual situations better than the constant volatility model, such as describing the phenomena of the volatility clustering and the heavy tailedness of return distributions \cite{Cui2017, He2023, Lu2021}. Therefore, it is desirable to investigate the investment and consumption problems under the stochastic volatility model \cite{Kang2021, Liu2007, Pham2009, Zhang2016}. Amongst the aforementioned literature, the dynamic programming approach (or HJB approach) is one of the most mainly used method to deal with the investment and consumption problems. Specially, Liu \cite{Liu2007} considered dynamic consumption and portfolio choice problems with quadratic asset returns and the power utility preference, and mentioned that when adopting the form of an exponential-polynomial value function, the polynomial with orders greater than $2$ cannot help in solving the HJB equation, while no proof was provided. Recently, Cheng and Escobar-Anel \cite{Cheng2023} investigated the portfolio optimization solvable problem for the investor with the power utility preference under stochastic volatility processes  and tried to prove the limitation on the order of polynomial for the conjecture of value function. After checking the proof process of Proposition 4.1 in \cite{Cheng2023},  we find that there is a gap in the derivation of equations (33) to (34) because the terms with $v^{4}$ in (33) depend on the order values of $\lambda^{2}(v)$, $m_{1}(v)$, $\lambda(v)m_{2}(v)$ and $m^{2}_{2}(v)$ with respect to $v$.  Thus, how to prove the limitation on the order of polynomial for the conjecture of value function in the investment and consumption problem under the stochastic volatility and the power utility preference is still worth further research.

On the other hand, the recursive utility functions,  being used to separate the relative risk aversion from elasticity of intertemporal substitution (EIS), have been widely adopted to capture the investor's consumption and investment preferences \cite{Epstein1989, Kreps1978}. As a continuous time limit of the recursive utility function, the stochastic differential utility (SDU), generalized form of power utility preference, has been received more and more attention because it can be used to more clearly connect the manner in which uncertainty is resolved over time with the ability to infer axiomatic differences in preferences by observing actions \cite{Duffie1992}. For example, Schroder and Skiadas \cite{Schroder1999} developed the utility gradient (or martingale) approach to considering portfolio and consumption problem with SDU in finite time region. Chacko and Viceira \cite{Chacko2005} applied Bellman's dynamic programming principle to study optimal portfolio choice and consumption with SDU under inverse Heston stochastic volatility in infinite time region, and derived analytical solutions for unit EIS and adopted the log-linear approximation method to obtain approximate solutions for non-unit EIS. Kraft et al. \cite{Kraft2013} provided verification theorems for continuous-time optimal consumption and investment with Epstein-Zin recursive preferences in incomplete markets. They also obtained analytical solutions for non-unit EIS under the particular configuration of preference parameters and compared analytical solutions to approximate solutions under the log-linear approximation. It is worth noting that the work mentioned above adopted the conjecture of the exponential-polynomial form of the value function to solve optimal consumption and investment problems with recursive preferences and showed that the HJB equation has analytical or approximate solutions when the order of the polynomial $n=1$ or $n=2$.  What happens to the analytical or approximate solutions of the HJB equation if the order of the polynomial $n>2$? However, to the best of our knowledge, this has not been considered in the literature.

Inspired and motivated by the above research works, the present paper is thus devoted to the study of the consumption and portfolio optimization problems with SDU in a financial market governed by a general stochastic volatility model within both infinite and finite time regions, in which the conjecture of the exponential-polynomial form of the value function is employed to solve such problems. The main contributions of this paper are threefold. First, the method of changing control is used to reveal solvable models with more complex expressions for the drift and diffusion terms of the risky asset and the general stochastic volatility. Second, with the conjecture of the exponential-polynomial form of the value function under mild conditions, some new results are obtained for the optimal consumption-investment with SDU under stochastic volatility model as follows: (i) when the order of the polynomial $n\leq2$, the HJB equation has an analytical solution if the investor with unit EIS and an approximate solution by the log-linear approximation method otherwise; (ii) when the order of the polynomial $n>2$, the HJB equation has no analytical solutions or approximation solutions by the log-linear approximation method, which appears for the first time to provide complete proofs in the literature for consumption and portfolio optimization problems with recursive preferences and the stochastic volatility. Third, some applications of our study to Heston's model are provided. Specially, expressions for the analytical and approximate solutions are given to demonstrate the solvability analysis and some numerical experiments are presented to illustrate the behavior of the optimal consumption-portfolio strategies.

The rest of this paper is organized as follows. Section 2 introduces the financial market and related assumptions. Section 3 formulates the consumption and portfolio problems with recursive preferences. Section 4 provides the solvability analysis of consumption and portfolio problems. Section 5 presents some applications of our research to Heston's model and provides some numerical experiments. Section 6 gives some conclusions and appendices contain the proofs.

\section{Preliminaries and assumptions}
Let $(\Omega,\mathcal{F},\mathbb{F},\mathbb{P})$ be a complete filtered probability space equipped with a natural filtration $\mathbb{F}=\{\mathcal{F}_{t}\}_{t\geq0}$ satisfying the usual conditions and a physical probability measure $\mathbb{P}$, where $\mathbb{F}$ is generated by 2-dimensional Brownian motion $(W^{\nu},W^{\perp})$ that will be defined later. Let $O_{0}\subset\mathbb{R}^{2}$ be an open set and $O=[0,T]\times O_{0}$. Denote that
\begin{align*}
C^{1,2}(O)=&\left\{\varphi(t,x)|\varphi(t,\cdot)\; \text{is once continuously differentiable on} \;[0,T]\; \text{and}\right. \\
 &\;\left.\varphi(\cdot,x)\;\text{is twice continuously differentiable on}\; O_{0} \right\}.
\end{align*}

We consider a frictionless financial market consisting of a risk-free asset and a risky asset that can be traded continuously. Specifically, the price $M_{t}$ of the risk-free asset follows
\begin{equation}\label{eq1}
  \mathrm{d}M_{t}=rM_{t}\mathrm{d}t, \quad M_{0}=1,
\end{equation}
where $r$ is a positive constant that denotes the risk-free interest rate. Moreover, inspired by the studies \cite{Cheng2023, Kraft2017}, we assume that the price $S_{t}$ of the risky asset satisfies the following general model
\begin{equation}\label{eq2}
\frac{\mathrm{d}S_{t}}{S_{t}}=\left[r+\eta(\nu_{t})G(\nu_{t},t)\right]\mathrm{d}t+G(\nu_{t},t)\mathrm{d}W^{S}_{t}, \quad S_{0}=s_{0}>0
\end{equation}
with a state variable $\nu_{t}$ evolving as follows
\begin{equation}\label{eq3}
\mathrm{d}\nu_{t}=m_{1}(\nu_{t})\mathrm{d}t+m_{2}(\nu_{t})\mathrm{d}W^{\nu}_{t}, \quad \nu_{0}=\overline{\nu}>0,
\end{equation}
where $\eta(\nu_{t})$ represents the market price of risk, $G(\nu_{t},t)$ denotes the volatility of $S_{t}$,  and $m_{1}(\nu_{t})$ and $m_{2}(\nu_{t})$ are the drift and volatility of $\nu_{t}$. The correlation between $S_{t}$ and $\nu_{t}$ can be captured by $W^{S}_{t}$ and $W^{\nu}_{t}$ via the parameter $\rho\in[-1,1]$. Therefore, we can write $\mathrm{d}W^{S}_{t}=\rho\mathrm{d}W^{\nu}_{t}+\sqrt{1-\rho^{2}}\mathrm{d}W^{\perp}_{t}$, where $W^{\perp}$ is another standard Brownian motion independent of $W^{\nu}$.

Following the work \cite{Cheng2023}, we now make the following assumptions:
\begin{itemize}
  \item[(A.1)] All the coefficients of equations \eqref{eq2} and \eqref{eq3} are progressively measurable with respect to $\{\mathcal{F}_{t}\}_{t\geq0}$.
  \item[(A.2)] To guarantee the uniqueness of the solution to \eqref{eq2}, we assume (see, for example, \cite{Korn2001, Kraft2005})
\begin{equation}\label{eq4}
\int_{0}^{\infty}\left(|\eta(\nu_{t})G(\nu_{t},t)|+G^{2}(\nu_{t},t)\right)\mathrm{d}t<\infty\quad a.s.
\end{equation}
  \item[(A.3)] To obtain the existence of the solution to \eqref{eq3}, the growth condition on the coefficients of \eqref{eq3} holds. That is,
\begin{equation}\label{eq5}
m_{1}^{2}(x_{0})+m_{2}^{2}(x_{0})\leq K(1+x^{2}_{0})
\end{equation}
for $x_{0}\in\mathbb{R}$ and $|x_{0}|\leq K_{1}$ with positive constants $K$ and $K_{1}$.
  \item[(A.4)] To ensure the uniqueness of the solution to \eqref{eq3}, the Yamada-Watanabe condition (see, for example, Theorem 4 in \cite{Yamada1971}) holds. That is, there exist real-valued functions $\widetilde{f}(x)$ and $g(x)$ defined on $[0, K_{2})$ with $ K_{2}>0$ such that
\begin{align}\label{eq6}
|m_{1}(x_{0})-m_{1}(y_{0})|\leq \widetilde{f}(|x_{0}-y_{0}|), \quad
|m_{2}(x_{0})-m_{2}(y_{0})|\leq g(|x_{0}-y_{0}|)
\end{align}
for all $x_{0},y_{0}\in\mathbb{R}$ with $|x_{0}-y_{0}|<K_{2}$, where $\widetilde{f}$ and $g$ are continuous, positive and increasing with $\widetilde{f}(0)=g(0)=0$, and $\widetilde{f}(x)$ and $g^{2}(x)x^{-1}$ are concave and satisfy
\begin{equation}\label{eq7}
\int_{0+}\left[\widetilde{f}(x)+g^{2}(x)x^{-1}\right]^{-1}\mathrm{d}x=\infty.
\end{equation}
\end{itemize}

\section{Problem formulation}
In this section, we consider a rational investor who can invest in the aforementioned financial market. We denote by $\pi_{t}$ the proportion of wealth invested in the risky asset and $(1-\pi_{t})$ the remaining proportion of wealth invested in the risk-free asset. Moreover, the investor can consume at an instantaneous rate $c_{t}$ at time $t$. Therefore, given the initial wealth value $X_{0}$, the wealth dynamics of the investor is described by the following stochastic differential equation (SDE)
\begin{equation}\label{eq8}
\mathrm{d}X_{t}=\left[r+\pi_{t}\eta(\nu_{t})G(\nu_{t},t)\right]X_{t}\mathrm{d}t-c_{t}\mathrm{d}t
+\pi_{t}G(\nu_{t},t)X_{t}\mathrm{d}W^{S}_{t}.
\end{equation}

Now, we assume that the preferences of the investor are captured by recursive utility functions, which are also known as stochastic differential utilities in continuous time. Due to \cite{Chacko2005, Duffie1992}, the investor's preference can be given by
\begin{equation}\label{eq9}
J(t,x,c,\psi)=J_{t}=\mathbb{E}_{t}\left[\int^{\infty}_{t}f(c_{s},J_{s})\mathrm{d}s\right].
\end{equation}
Here, $\mathbb{E}_{t}$ represents the $\mathcal{F}_{t}$-conditional expectation with respect to the measure $\mathbb{P}$ and $f(c,J)$ denotes the Epstein-Zin aggregator of consumption and continuation value that takes the form
\begin{equation}\label{eq10}
f(c,J)=\beta(1-\frac{1}{\phi})^{-1}(1-\gamma)J\left[\left(\frac{c}{((1-\gamma)J)^{\frac{1}{1-\gamma}}}\right)^{1-\frac{1}{\phi}}-1\right],
\end{equation}
where $\beta>0$ denotes the investor's discount rate, $\gamma>0$ with $\gamma\neq1$ represents the relative risk aversion coefficient and $\phi>0$ with $\phi\neq1$ is the EIS parameter. When $\phi=1$, the aggregator $f(c,J)$ takes the form
\begin{equation}\label{eq11}
f(c,J)=\beta(1-\gamma)J\left[\ln c-\frac{1}{1-\gamma}\ln ((1-\gamma)J)\right].
\end{equation}
If $\phi=\frac{1}{\gamma}$, then $f$ takes the form of the power utility. Moreover, if $\phi=\gamma=1$, then $f$ takes the form of the logarithmic utility \cite{Wei2023}.

Next, given the wealth dynamics \eqref{eq8} and the recursive preference \eqref{eq10} or \eqref{eq11}, the investor aims to choose a consumption-investment strategy $(c_{t},\pi_{t})$ to maximize the recursive utility
\begin{align}\label{eq12}
\sup\limits_{(c_{t},\pi_{t})\in \widetilde{\Pi}_{1}} \; \mathbb{E}_{t}\left[\int^{\infty}_{t}f(c_{s},J_{s})\mathrm{d}s\right],
\end{align}
where $\widetilde{\Pi}_{1}$ denotes the set of admissible strategies as defined below.
\begin{definition}\label{definition3.1}
For $t\in[0,\infty)$, a consumption-investment strategy $(c_{t},\pi_{t})$ of \eqref{eq12} is admissible if the following conditions are satisfied:
\begin{itemize}
\item[(i)] $c_{t}$ and $\pi_{t}$ are $\{\mathcal{F}_{t}\}_{t\geq0}$-adapted feedback strategies $c=\{c_{t}\}_{t\geq0}=\{c(t,X_{t},\nu_{t})\}_{t\geq0}$ and $\pi=\{\pi_{t}\}_{t\geq0}=\{\pi(t,X_{t},\nu_{t})\}_{t\geq0}$ taking value in $(0,\infty)$ and $\mathbb{R}$, respectively;
\item[(ii)] for any initial values $(t, x, \nu)\in [0,\infty) \times \mathbb{R}^{+} \times \mathbb{R}$, the SDE $\eqref{eq8}$ for $\{X_{s}\}_{s\geq t}$ with $X_{t}=x$ and $\nu_{t}=\nu$ admits a unique positive strong solution;
\item[(iii)] $\mathbb{E}_{t}\left[\int^{\infty}_{t}\left|f(c_{s},J_{s})\right|\mathrm{d}s\right] <\infty$.
\end{itemize}
\end{definition}

For the sake of convenience, we denote a new control variable by
\begin{align}\label{eq13}
\psi_{t}=\pi_{t}G(\nu_{t},t)
\end{align}
and the corresponding admissible set by $\Pi_{1}$. Therefore, the wealth process \eqref{eq8} and problem \eqref{eq12} in term of the new control variable can be rewritten as
\begin{align}\label{eq14}
\mathrm{d}X_{t}=\left[r+\eta(\nu_{t})\psi_{t}\right]X_{t}\mathrm{d}t-c_{t}\mathrm{d}t+\psi_{t}X_{t}\mathrm{d}W^{S}_{t}
\end{align}
and
\begin{align}\label{eq15}
\sup\limits_{(c_{t},\psi_{t})\in \Pi_{1}} \; \mathbb{E}_{t}\left[\int^{\infty}_{t}f(c_{s},J_{s})\mathrm{d}s\right],
\end{align}
respectively. It seems that the equation \eqref{eq14} under $\psi_{t}$ looks simpler than \eqref{eq8}. In addition, if the problem \eqref{eq15} is solved in term of $(c_{t}, \psi_{t})$, then we can obtain the solution to the original problem \eqref{eq12}. In order to make sure the objective function \eqref{eq15} is well-defined, we suppose that
\begin{itemize}
  \item[(B.1)] for any $c\in(0,\infty)$ and $\psi\in\mathbb{R}$ with $\psi=\{\psi_{t}\}_{t\geq0}$, the equation \eqref{eq9} has a unique strong solution.
\end{itemize}

When the investment period is a finite interval, we follow \cite{Kraft2017, Kraft2013} to rewrite the investor's preference in \eqref{eq9} as
\begin{equation}\label{eq16}
J_{t}=\mathbb{E}_{t}\left[\int^{T}_{t}f(c_{s},J_{s})\mathrm{d}s+U(X_{T})\right],
\end{equation}
where $T>0$ denotes the terminal of investment interval, $U:(0,\infty)\rightarrow\mathbb{R}$ with $U(x)=\varepsilon^{1-\gamma}\frac{x^{1-\gamma}}{1-\gamma}$ is a constant relative risk aversion utility function for bequest, and where $\varepsilon\in(0,\infty)$ is a weight factor. Then the investor wants to adopt the control variable $(c_{t},\psi_{t})$ to maximize the preference \eqref{eq16} as follows
\begin{align}\label{eq17}
\sup\limits_{(c_{t},\psi_{t})\in \Pi_{2}} \; \mathbb{E}_{t}\left[\int^{T}_{t}f(c_{s},J_{s})\mathrm{d}s+U(X_{T})\right].
\end{align}
Here, $\Pi_{2}$ denotes the set of admissible strategies, which can be defined similarly as Definition \ref{definition3.1}. Moreover, in order to ensure the objective function \eqref{eq17} is well-defined, we assume that
\begin{itemize}
  \item[(B.2)] for any $c\in(0,\infty)$ and $\psi\in\mathbb{R}$, the equation \eqref{eq16} has a unique strong solution.
\end{itemize}
\begin{remark}
We would like to point out that the equations \eqref{eq9} and \eqref{eq16} can be rewritten as the following backward stochastic differential equation (BSDE)
\begin{equation}\label{eq84}
   \left\{ \begin{aligned}
  \mathrm{d}J_{s} =&-f(c_{s},J_{s})\mathrm{d}s+\widetilde{Z}_{s}^{1}\mathrm{d}W^{S}_{s}+\widetilde{Z}_{s}^{2}\mathrm{d}W^{\nu}_{s},\;\;s\geq t\\
   J_{T}=&U(X_{T})
  \end{aligned}\right.
\end{equation}
with $U(X_{T})\rightarrow0$ as $T\rightarrow\infty$. Thus, studying the unique strong solution of \eqref{eq9} or \eqref{eq16} can be also transformed into considering the unique strong solution of \eqref{eq84}. Interested readers may refer to Proposition 2.2 and Theorem 2.5 in \cite{Aurand2023}, Theorem B.2 in \cite{Herdegen2023}, Corollary 2.3 in \cite{Melnyk2020} and Proposition 2.2 in \cite{Xing2017} for more details on the uniqueness of recursive utility and the corresponding BSDE.
\end{remark}

\section{Optimal consumption and investment solvable problems}
In this section, we will adopt Bellman's dynamic programming principle to solve problems \eqref{eq15} and \eqref{eq17}, respectively. We note that a verification theorem is crucial to derive the HJB equation and to make sure that the solution to the HJB equation is indeed
the solution to the problem under consideration. By the fact that the aggregator $f(c,J)$ in \eqref{eq10} and \eqref{eq11} is non-Lipschitz, standard verification results (see, for example, \cite{Duffie1992}) are not applicable. Kraft et al. provided new verification theorems through Corollary B.1 and Theorem 3.1 in  \cite{Kraft2013} to consumption-portfolio problems with recursive utility under incomplete markets within infinite and finite time regions, respectively. To deal with the non-Lipschitz property of $f(c,J)$, we follow \cite{Kraft2013} to assume that there exists $L>0$ such that the aggregator $f(c,J)$ in \eqref{eq10} satisfies
\begin{itemize}
  \item[(B.3)] $f(c,J_{1})-f(c,J_{2})\leq L (J_{1}-J_{2})$ for all $c\in(0,\infty)$, $J_{1}$ and $J_{2}$ with $J_{1}\geq J_{2}$.
\end{itemize}
On the other hand, Fan et al. \cite{Fan2011} established a general comparison theorem for solutions of one-dimensional BSDEs with weakly monotonic generators. Inspired by \cite{Fan2011}, we make the following assumption for the aggregator $f(c,J)$ in \eqref{eq11} that
\begin{itemize}
  \item[(B.4)] there exists a nondecreasing concave continuous function $\chi(\cdot)$ from $\mathbb{R}^{+}$ to itself with $\chi(0)=0$ and $\chi(y)>0$ for $y>0$ such that $\int_{0^{+}}\frac{1}{\chi(y)}\mathrm{d}y=\infty$ and for all $c$, $J_{1}$ and $J_{2}$,
      $sgn(J_{1}-J_{2})\times\left[f(c,J_{1})-f(c,J_{2})\right]\leq \chi(\left|J_{1}-J_{2}\right|)$ a.s..
\end{itemize}
In the sequel, we will utilize the conditions (B.3) and (B.4) to give the corresponding verification theorem.

\subsection{The case of infinite time-horizon}
 First of all, we consider the case of the infinite time interval. When the investor has unit EIS, which implies the aggregator $f(c,J)$ takes the form \eqref{eq11} in the problem \eqref{eq15}, we assume that
\begin{itemize}
  \item[(B.5)] $c\in[C_{m},C_{M}]$ with $0<C_{m}<C_{M}<\infty$ and for any $c\in[C_{m},C_{M}]$ and $\psi\in\mathbb{R}$, $\mathbb{E}[J_{T}]\rightarrow0$ and $\mathbb{E}[\omega(X_{T})]\rightarrow0$ as $T\rightarrow\infty$, where $J_{T}$ is given by \eqref{eq9};
  \item[(B.6)] all coefficients are time-homogeneous and for any $c\in[C_{m},C_{M}]$ and $\psi\in\mathbb{R}$,
      $\int^{\cdot}_{t}\omega_{x}(X_{s})\psi_{s}X_{s}\sqrt{1-\rho^{2}}\mathrm{d}W^{\perp}_{s}$ and $\int^{\cdot}_{t}\left[\omega_{x}(X_{s})\psi_{s}X_{s}\rho+\omega_{\nu}(X_{s})m_{2}(\nu_{s})\right]\mathrm{d}W^{\nu}_{s}$ are uniformly integrable martingales.
\end{itemize}
\noindent Then, it follows from \cite{Chacko2005, Kraft2013} that the corresponding HJB equation can be given by
\begin{equation}\label{eq18}
\begin{aligned}
\sup\limits_{c\in[C_{m},C_{M}],\psi\in\mathbb{R}} \;&\left\{\left[r+\psi\eta(\nu)\right]x\omega_{x}-c\omega_{x}+\frac{1}{2}\psi^{2}x^{2}\omega_{xx}+m_{1}(\nu)\omega_{\nu}+\frac{1}{2}m^{2}_{2}(\nu)\omega_{\nu\nu}\right.\\
&\left.\quad+x\psi\rho m_{2}(\nu)\omega_{x\nu}+\beta(1-\gamma)\omega\left[\ln c-\frac{1}{1-\gamma}\ln ((1-\gamma)\omega)\right]\right\}=0,
\end{aligned}
\end{equation}
where $\omega$ denotes the candidate value function of the problem \eqref{eq15} when taking the form \eqref{eq11}, and $\omega_{x}$, $\omega_{\nu}$, $\omega_{xx}$, $\omega_{\nu\nu}$ and $\omega_{x\nu}$ represent the first and second partial derivatives of $\omega$ with respect to $x$ and $\nu$.
We will adopt the comparison theorem of BSDE under the weakly monotonic condition to give the verification theorem as follows.

\begin{theorem}\label{theorem4.5}
Assume that the conditions (B.1), (B.5) and (B.6) are satisfied. If $\omega\in C^{2}(\mathbb{R}^{+}\times\mathbb{R})$ is a solution of the HJB equation \eqref{eq18} and
\begin{align}\label{eq83}
(c^{*},\psi^{*})=\arg\max\limits_{c\in[C_{m},C_{M}],\psi\in\mathbb{R}} \;&\left\{\left[r+\psi\eta(\nu)\right]x\omega_{x}-c\omega_{x}+\frac{1}{2}\psi^{2}x^{2}\omega_{xx}+m_{1}(\nu)\omega_{\nu}+\frac{1}{2}m^{2}_{2}(\nu)\omega_{\nu\nu}\right.\nonumber\\
&\left.\quad+x\psi\rho m_{2}(\nu)\omega_{x\nu}+\beta(1-\gamma)\omega\left[\ln c-\frac{1}{1-\gamma}\ln ((1-\gamma)\omega)\right]\right\},
\end{align}
then $c^{*}$ and $\psi^{*}$ are optimal and $\omega$ coincides with the value function of the problem \eqref{eq15} when the aggregator $f(c,J)$ takes the form \eqref{eq11}.
\end{theorem}
\begin{proof}
Letting $c_{0}=(1-\gamma)\ln c-\ln(1-\gamma)$ and
\begin{align}\label{eq80}
\chi(y)=
\begin{cases}
0, \quad\quad\quad\quad\quad\;\quad\quad\;\text{for}\; \quad\, y=0,\\
-y(\ln y-c_{0}), \quad\quad\;\text{for}\;\quad  0<y\leq e^{-2+c_{0}},\\
y+e^{-2+c_{0}},  \quad\quad\quad\;\text{for}\;\quad  y> e^{-2+c_{0}},
\end{cases}
\end{align}
we have $\frac{1}{\beta}\left[f(c,J_{1})-f(c,J_{2})\right]=c_{0}(J_{1}-J_{2})-(J_{1}\ln J_{1}-J_{2}\ln J_{2})$.
By the condition (B.5), we know that $c\in[C_{m},C_{M}]$ and then $c_{0}$ is bounded. Moreover, it follows from the equation \eqref{eq80} that $\chi(y)$ is a nondecreasing concave continuous function from $\mathbb{R}^{+}$ to itself with $\chi(0)=0$ and $\chi(y)>0$ for $y>0$. Furthermore, it is not hard to verify that $\int_{0^{+}}\frac{1}{\chi(y)}\mathrm{d}y=\infty$.

Next, we will show that when $c\in[C_{m},C_{M}]$, the condition (B.4) holds for the aggregator $f(c,J)$ in \eqref{eq11}. Firstly, we consider the case (1): $J_{1},J_{2}> e^{-1+c_{0}}$.

(1a) If $J_{1}>J_{2}$, according to the monotonic decrease of $f(c,J)$ with respect to $J$ on $(e^{-1+c_{0}},\infty)$, then one has
$\frac{sgn(J_{1}-J_{2})}{\beta}\times\left[f(c,J_{1})-f(c,J_{2})\right]<0\leq\chi(\left|J_{1}-J_{2}\right|)$.

(1b) If $J_{1}<J_{2}$, similar to the case (1a), then one can derive
$\frac{sgn(J_{1}-J_{2})}{\beta}\times\left[f(c,J_{1})-f(c,J_{2})\right]<0\leq\chi(\left|J_{1}-J_{2}\right|)$.

Secondly, we consider the case (2): $0\leq J_{1}\leq e^{-1+c_{0}}<J_{2}$.

(2a) If $f(c,J_{1})\geq f(c,J_{2})$, then we can obtain that $\frac{sgn(J_{1}-J_{2})}{\beta}\times\left[f(c,J_{1})-f(c,J_{2})\right]\leq0\leq\chi(\left|J_{1}-J_{2}\right|)$.

(2b) If $f(c,J_{1})< f(c,J_{2})$, then we have that $\frac{f(c,J_{2})-f(c,J_{1})}{J_{2}-J_{1}}\leq\frac{f(c,e^{-1+c_{0}})-f(c,J_{1})}{e^{-1+c_{0}}-J_{1}}\leq\frac{f(c,e^{-1+c_{0}})}{e^{-1+c_{0}}}=\beta$. Moreover, since $\chi_{y}(y)\geq1$, one has $\chi(y)\geq y$. Thus, one has $\frac{sgn(J_{1}-J_{2})}{\beta}\times\left[f(c,J_{1})-f(c,J_{2})\right]=\frac{f(c,J_{2})-f(c,J_{1})}{\beta}\leq J_{2}-J_{1}\leq\chi(\left|J_{1}-J_{2}\right|)$.
Similarly, the results for the case (3): $0\leq J_{2}\leq e^{-1+c_{0}}<J_{1}$ can be also derived.

Finally, we consider the case (4): $0\leq J_{1},J_{2}\leq e^{-1+c_{0}}$.

(4a) If $J_{1}>J_{2}$, then
\begin{align*}
  &\frac{sgn(J_{1}-J_{2})}{\beta}\times\left[f(c,J_{1})-f(c,J_{2})\right]\\
  &=\frac{f(c,J_{1})-f(c,J_{2})}{\beta}\nonumber\\
  &=\int^{J_{1}}_{J_{2}}\left[-1-\ln y +c_{0}\right]\mathrm{d}y\\
  &\leq \int^{J_{1}-J_{2}}_{0}\left[-1-\ln y +c_{0}\right]\mathrm{d}y\\
  &=-y(\ln y-c_{0})|^{J_{1}-J_{2}}_{0}\\
  &=-(J_{1}-J_{2})\left[\ln(J_{1}-J_{2})-c_{0}\right]\\
  &\leq\chi(\left|J_{1}-J_{2}\right|),
\end{align*}
where we apply the monotonic decreasing property of function $-1-\ln y +c_{0}$ on $(0,\infty)$ in the third inequality.

(4b) If $J_{1}<J_{2}$, similar to the case (4a), then we can derive that $\frac{sgn(J_{1}-J_{2})}{\beta}\times\left[f(c,J_{1})-f(c,J_{2})\right]\leq\chi(\left|J_{1}-J_{2}\right|)$.

Now, applying the It\^{o} formula to $\omega(t,\nu_{t},X_{t})$, one has
\begin{align}\label{eq81}
  \mathrm{d}\omega(s,\nu_{s},X_{s}) =&\left\{\left[r+\psi_{s}\eta(\nu_{s})\right]X_{s}\omega_{x}(s,\nu_{s},X_{s})-c_{s}\omega_{x}(s,\nu_{s},X_{s})+\frac{1}{2}\psi^{2}_{s}X^{2}_{s}\omega_{xx}(s,\nu_{s},X_{s})
   \right.\nonumber\\
   &\left.+m_{1}(\nu_{s})\omega_{\nu}(s,\nu_{s},X_{s})+\frac{1}{2}m^{2}_{2}(\nu_{s})\omega_{\nu\nu}(s,\nu_{s},X_{s})
+X_{s}\psi_{s}\rho m_{2}(\nu_{s})\omega_{x\nu}(s,\nu_{s},X_{s})\right\}\mathrm{d}s\nonumber\\
&+\psi_{s}X_{s}\omega_{x}(s,\nu_{s},X_{s})\mathrm{d}W^{S}_{s}+m_{2}(\nu_{s})\omega_{\nu}(s,\nu_{s},X_{s})\mathrm{d}W^{\nu}_{s}
\end{align}
with $\omega(T,\nu_{T},X_{T})\rightarrow0$ as $T\rightarrow\infty$. Under the condition (B.1), the definition of $J_{t}$ in \eqref{eq15} with the aggregator $f(c,J)$ given by \eqref{eq11} can be rewritten as the following BSDE form
\begin{align}\label{eq82}
  \mathrm{d}J_{s} =&-\beta(1-\gamma)J_{s}\left[\ln c-\frac{1}{1-\gamma}\ln ((1-\gamma)J_{s})\right]\mathrm{d}s+Z_{s}^{1}\mathrm{d}W^{S}_{s}+Z_{s}^{2}\mathrm{d}W^{\nu}_{s}
\end{align}
with $J_{T}\rightarrow0$ as $T\rightarrow\infty$, where $\mathbb{E}_{t}\left[\int^{\infty}_{t}|Z_{s}^{1}|^{2}\mathrm{d}s\right]<\infty$ and $\mathbb{E}_{t}\left[\int^{\infty}_{t}|Z_{s}^{2}|^{2}\mathrm{d}s\right]<\infty$. By the comparison theorem for BSDEs (see, Theorem 2 in \cite{Fan2011}), it follows from the condition (B.5) and the equations \eqref{eq18}, \eqref{eq81} and \eqref{eq82} that for any $t$, $J_{t}\leq \omega(t,\nu_{t},X_{t})$ a.s.. Therfore, $J(t,x,c,\psi)\leq \omega(t,\nu,x)$ for any $t$. Since $c$ and $\psi$ are arbitrary, one can obtain that $\sup\limits_{c,\psi} \;J(t,x,c,\psi)\leq\omega(t,\nu,x)$. On the other hand, by using the equation \eqref{eq83}, we derive that $J(t,x,c^{*},\psi^{*})=\omega(t,\nu,x)$. Thus, $c^{*}$ and $\psi^{*}$ are optimal and $\omega$ is the value function of the problem \eqref{eq15} when the aggregator $f(c,J)$ takes the form \eqref{eq11}.
\end{proof}

By differentiating the expression inside the bracket of $\eqref{eq18}$ with respect to $c$ and $\psi$, respectively, one can obtain
\begin{equation}\label{eq19}
   \left\{ \begin{aligned}
   &c^{*}_{t}=\frac{\beta(1-\gamma)\omega}{\omega_{x}},\\
   &\psi^{*}_{t}=-\frac{\eta(\nu)\omega_{x}+\rho m_{2}(\nu)\omega_{x\nu}}{x\omega_{xx}}.
  \end{aligned}\right.
\end{equation}
We make an ansatz that $\omega(x,\nu)=\frac{x^{1-\gamma}}{1-\gamma}h(\nu)^{1-\gamma}$ for some deterministic function $h(\nu)$. Substituting $\eqref{eq19}$ and the ansatz for $\omega(x,\nu)$ into $\eqref{eq18}$ yields the following partial differential equation (PDE)
\begin{align}\label{eq20}
  r&-\beta+\frac{1}{2\gamma}\eta^{2}(\nu)+\beta\left(\ln\beta-\ln h\right)+\left[m_{1}(\nu)+\frac{\eta(\nu)}{\gamma}\rho m_{2}(\nu)(1-\gamma)\right]\frac{h_{\nu}}{h}\nonumber\\
  &+\frac{ 1}{2}m^{2}_{2}(\nu)\left[\frac{\rho^{2}(1-\gamma)^{2}}{\gamma}-\gamma\right]\frac{h^{2}_{\nu}}{h^{2}}+\frac{ 1}{2}m^{2}_{2}(\nu)\frac{h_{\nu\nu}}{h}=0.
\end{align}
We suppose that
\begin{itemize}
  \item[(B.7)] $\eta^{2}(\nu)$ is quadratic in $\nu$, $m_{1}(\nu)$ and $\eta(\nu) m_{2}(\nu)$ is linear in $\nu$, and $m^{2}_{2}(\nu)$ is a constant.
\end{itemize}
In what follows, similar to the discussions in \cite{Cheng2023}, we will make a general ansatz for $h$ as an exponential-polynomial form and prove that when the order of the polynomial is higher than 2, the PDE \eqref{eq20} has no analytical solutions under the condition (B.7).
\begin{theorem}\label{theorem4.1}
If the condition (B.7) is satisfied and $h(\nu)$ is conjectured in the exponential-polynomial form such that
\begin{align}\label{eq21}
\omega(x,\nu)=\frac{x^{1-\gamma}}{1-\gamma}h(\nu)^{1-\gamma}=\frac{x^{1-\gamma}}{1-\gamma}\exp\left\{(1-\gamma)\left(A_{0}+\sum\limits_{k=1}^{n}\frac{1}{k}A_{k}\nu^{k}\right)\right\},
\end{align}
where $A_{k}$ for all $k=0,\cdot\cdot\cdot,n$ are constants, then the PDE \eqref{eq20} has no analytical solutions when the order $n$ of the polynomial is higher than 2.
\end{theorem}
\begin{proof}
In the following proof process, we first consider the case when $n=3$ and then discuss the case when $n>3$.

\noindent\textbf{Exponential-cubic:} Assume that $h(\nu)=\exp\{A_{0}+A_{1}\nu+\frac{1}{2}A_{2}\nu^{2}+\frac{1}{3}A_{3}\nu^{3}\}$.

Inserting the partial derivatives of $h$ into \eqref{eq20}, one has
\begin{align}\label{eq22}
  r&-\beta+\frac{1}{2\gamma}\eta^{2}(\nu)+\beta(\ln\beta-A_{0}-A_{1}\nu-\frac{1}{2}A_{2}\nu^{2}-\frac{1}{3}A_{3}\nu^{3})+\left[m_{1}(\nu)+\frac{\eta(\nu)}{\gamma}\rho m_{2}(\nu)(1-\gamma)\right](A_{1}+A_{2}\nu+A_{3}\nu^{2})\nonumber\\
  &+\frac{ 1}{2}m^{2}_{2}(\nu)\left[(1-\gamma)(A_{1}+A_{2}\nu+A_{3}\nu^{2})^{2}+A_{2}+2A_{3}\nu\right]
  +\frac{\rho^{2}}{2\gamma}m^{2}_{2}(\nu)(1-\gamma)^{2}(A_{1}+A_{2}\nu+A_{3}\nu^{2})^{2}=0.
\end{align}
Note that the above equation has terms involving $\nu^{4}$. For this equation to be solvable, one needs to cancel the terms involving $\nu^{4}$. It implies that under the condition (B.7),
\begin{equation}\label{eq23}
  1-\gamma+\frac{\rho^{2}}{\gamma}(1-\gamma)^{2}=0
\end{equation}
should be satisfied, which is equivalent to $\rho^{2}=\frac{\gamma}{\gamma-1}$ due to $\gamma\neq1$. If $\gamma>1$, then $\rho^{2}>1$, which contradicts the fact that $\rho\in[-1,1]$. If $0<\gamma<1$, then $\rho^{2}<0$, which is impossible. Therefore, the terms involving $\nu^{4}$ in \eqref{eq22} cannot be matched. Thus, the conjecture of the exponential-cubic form of $h$ cannot solve the PDE \eqref{eq20}.

\noindent\textbf{Exponential-nth:} Assume that $h(\nu)=\exp\{A_{0}+\sum\limits_{k=1}^{n}\frac{1}{k}A_{k}\nu^{k}\}$ with $n>3$.

By substituting the partial derivatives of $h$ into \eqref{eq20}, we have
\begin{align*}
  r&-\beta+\frac{1}{2\gamma}\eta^{2}(\nu)+\beta\left(\ln\beta-A_{0}-\sum\limits_{k=1}^{n}\frac{1}{k}A_{k}\nu^{k}\right)+\left[m_{1}(\nu)+\frac{\eta(\nu)}{\gamma}\rho m_{2}(\nu)(1-\gamma)\right]\sum\limits_{k=1}^{n}A_{k}\nu^{k-1}\nonumber\\
  &+\frac{ 1}{2}m^{2}_{2}(\nu)\left[(1-\gamma)\left(\sum\limits_{k=1}^{n}A_{k}\nu^{k-1}\right)^{2}+\sum\limits_{k=2}^{n}(k-1)A_{k}\nu^{k-2}\right]
  +\frac{\rho^{2}}{2\gamma}m^{2}_{2}(\nu)(1-\gamma)^{2}\left(\sum\limits_{k=1}^{n}A_{k}\nu^{k-1}\right)^{2}=0,
\end{align*}
which involves $\nu^{n+1},\cdot\cdot\cdot,\nu^{2n-2}$ for $n>3$. These terms cannot be matched unless $1-\gamma+\frac{\rho^{2}}{\gamma}(1-\gamma)^{2}=0$ under the condition (B.7). It follows from the previous analysis that the equation \eqref{eq23} does not hold. Therefore, if $n>3$, then the conjecture of the exponential-polynomial form of $h$ cannot solve the PDE \eqref{eq20}.
\end{proof}

\begin{remark}\label{remark6}
We would like to point out that the assumption $\eta^{2}(\nu)$ is quadratic in $\nu$ can be relaxed as $\eta^{2}(\nu)$ is at most a quadratic function of $\nu$ in the proof of Theorem \ref{theorem4.1}. The assumptions regarding $m_{1}(\nu)$ and $\eta(\nu) m_{2}(\nu)$ are similar. Moreover, it follows from the proof of Theorem \ref{theorem4.1} that there assumptions are crucial for the derivation of \eqref{eq23} and the unsolvable analysis of \eqref{eq22}. Without these assumptions, the analysis of whether \eqref{eq22} is solvable would be complex and it would be difficult to come to a conclusion. Such assumptions are also made in Theorems \ref{theorem4.2}, \ref{theorem4.3} and \ref{theorem4.4}. We observe that \cite{Cheng2023} did not make any assumptions regarding the order values of $\lambda^{2}(v)$, $m_{1}(v)$, $\lambda(v)m_{2}(v)$ and $m^{2}_{2}(v)$ with respect to $v$ and thus the derivation of (34) and the unsolvable analysis of (33) in \cite{Cheng2023} remain for further discussion. Therefore, our derivation seems more rigorous from the perspective of mathematics.
\end{remark}

Theorem \ref{theorem4.1} shows that the PDE \eqref{eq20} has no analytical solutions under the ansatz of an exponential-polynomial form of $h$ when the order of the polynomial $n$ is higher than 2. What happens to the analytical solutions of the PDE \eqref{eq20} if $n\leq2$? In the sequel, we attempt to consider the cases when $n=0,1,2$. We assume that
\begin{itemize}
  \item[(B.8)] $\eta(\nu)$ is a constant;
  \item[(B.9)] $\eta^{2}(\nu)$, $m_{1}(\nu)$, $\eta(\nu) m_{2}(\nu)$ and $m^{2}_{2}(\nu)$ are linear in $\nu$.
\end{itemize}

\begin{theorem}\label{theorem4.6}
Assume that $h$ is conjectured in the form of \eqref{eq21}. (i) If the condition (B.8) is satisfied, then the PDE \eqref{eq20} has an analytical solution when $n=0$; (ii) If the condition (B.9) is satisfied, then the PDE \eqref{eq20} has an analytical solution when $n=1$; (iii) If the condition (B.7) is satisfied, then the PDE \eqref{eq20} has an analytical solution when $n=2$.
\end{theorem}
\begin{proof}
In the sequel, we will provide the solvability of analytical solutions of the PDE \eqref{eq20} in terms of $\eta(\nu)$, $m_{1}(\nu)$ and $m_{2}(\nu)$ under the conjecture form \eqref{eq21} when $n=0,1,2$.

\noindent\textbf{Exponential-constant:} Assume that $h(\nu)=\exp\{A_{0}\}$.\\
Substituting the partial derivatives of $h$ into \eqref{eq20} yields
\begin{equation*}
  r-\beta+\frac{1}{2\gamma}\eta^{2}(\nu)+\beta(\ln\beta-A_{0})=0,
\end{equation*}
which implies that the PDE \eqref{eq20} is solvable under the condition (B.8).

\noindent\textbf{Exponential-linear:} Assume that $h(\nu)=\exp\{A_{0}+A_{1}\nu\}$.\\
Substituting the partial derivatives of $h$ into \eqref{eq20} leads to
\begin{align*}
  r&-\beta+\frac{1}{2\gamma}\eta^{2}(\nu)+\beta(\ln\beta-A_{0}-A_{1}\nu)+\left[m_{1}(\nu)+\frac{\eta(\nu)}{\gamma}\rho m_{2}(\nu)(1-\gamma)\right]A_{1}\\
  &+\frac{ 1}{2}m^{2}_{2}(\nu)\left[1-\gamma+\frac{\rho^{2}}{\gamma}(1-\gamma)^{2}\right]A_{1}^{2}=0,
\end{align*}
which means that the PDE \eqref{eq20} can be solvable under the condition (B.9).

\noindent\textbf{Exponential-quadratic:} Assume that $h(\nu)=\exp\{A_{0}+A_{1}\nu+\frac{1}{2}A_{2}\nu^{2}\}$.\\
Similarly, we can obtain that
\begin{align*}
  r&-\beta+\frac{1}{2\gamma}\eta^{2}(\nu)+\beta(\ln\beta-A_{0}-A_{1}\nu-\frac{1}{2}A_{2}\nu^{2})+\left[m_{1}(\nu)+\frac{\eta(\nu)}{\gamma}\rho m_{2}(\nu)(1-\gamma)\right](A_{1}+A_{2}\nu)\\
  &+\frac{ 1}{2}m^{2}_{2}(\nu)\left[(1-\gamma)(A_{1}+A_{2}\nu)^{2}+A_{2}\right]
  +\frac{\rho^{2}}{2\gamma}m^{2}_{2}(\nu)(1-\gamma)^{2}(A_{1}+A_{2}\nu)^{2}=0,
\end{align*}
which shows that the PDE \eqref{eq20} is solvable under the condition (B.7).
\end{proof}

\begin{remark}\label{remark1}
Under the assumption (B.7), Theorem \ref{theorem4.6} shows the PDE \eqref{eq20} has an analytical solution under the ansatz of an exponential-polynomial form of $h$ when $n=2$, which implies that the problem \eqref{eq15} with $f(c,J)$ given by \eqref{eq11} is solvable. Moreover, Theorem \ref{theorem4.1} proves that the PDE \eqref{eq20} has no analytical solutions under the ansatz of $h$ when $n>2$. Thus, Theorems \ref{theorem4.1} and \ref{theorem4.6} fully reveal the relationship between the order of polynomial for the conjecture of value function and the solution of the HJB equation under the same assumption. Furthermore, under the assumption (B.8) or (B.9), we can use for reference of the proof of Theorem \ref{theorem4.1} and combine Theorem \ref{theorem4.6} to derive similar conclusions.
\end{remark}

Next, we consider the case of $\phi\neq1$. That is, the aggregator $f(c,J)$ takes the form \eqref{eq10} in the problem \eqref{eq15}. Similarly, the corresponding HJB equation can be given as follows
\begin{equation}\label{eq24}
\begin{aligned}
\sup\limits_{c\in(0,\infty),\psi\in\mathbb{R}} \;&\left\{\left[r+\psi\eta(\nu)\right]x\omega_{x}-c\omega_{x}+\frac{1}{2}\psi^{2}x^{2}\omega_{xx}+m_{1}(\nu)\omega_{\nu}+\frac{1}{2}m^{2}_{2}(\nu)\omega_{\nu\nu}\right.\\
&\left.\quad+x\psi\rho m_{2}(\nu)\omega_{x\nu}+\beta(1-\frac{1}{\phi})^{-1}(1-\gamma)\omega\left[\left(\frac{c}{((1-\gamma)\omega)^{\frac{1}{1-\gamma}}}\right)^{1-\frac{1}{\phi}}-1\right]\right\}=0.
\end{aligned}
\end{equation}
For the sake of simplicity, we still use the symbol $\omega$. Similar expressions will be used in this section when there is no ambiguity. To adopt Corollary B.1 in \cite{Kraft2013} to ensure that the candidate value function $\omega$ in \eqref{eq24} is indeed the value function of the problem \eqref{eq15} when the aggregator $f(c,J)$ takes the form \eqref{eq10}, we want to ensure the condition (B.3). Thus, for the aggregator $f(c,J)$ with the form \eqref{eq10}, we follow Proposition 3.2 in \cite{Kraft2013} to assume that
\begin{itemize}
  \item[(B.10)] the coefficients $\gamma$ and $\phi$ satisfy one of the following four cases: (i) $\gamma>1$ and $\phi>1$, (ii) $\gamma>1$ and $\phi<1$ with $\gamma\phi\leq1$, (iii) $\gamma<1$ and $\phi<1$, (iv) $\gamma<1$ and $\phi>1$ with $\gamma\phi\geq1$.
\end{itemize}
\noindent Furthermore, we assume that
\begin{itemize}
  \item[(B.11)] all coefficients are time-homogeneous and for any $c\in(0,\infty)$ and $\psi\in\mathbb{R}$,
      $\int^{\cdot}_{t}\omega_{x}(X_{s})\psi_{s}X_{s}\sqrt{1-\rho^{2}}\mathrm{d}W^{\perp}_{s}$ and $\int^{\cdot}_{t}\left[\omega_{x}(X_{s})\psi_{s}X_{s}\rho+\omega_{\nu}(X_{s})m_{2}(\nu_{s})\right]\mathrm{d}W^{\nu}_{s}$ are uniformly integrable martingales;
  \item[(B.12)] for any $c\in(0,\infty)$ and $\psi\in\mathbb{R}$, $\mathbb{E}[J_{T}]\rightarrow0$ and $\mathbb{E}[\omega(X_{T})]\rightarrow0$ as $T\rightarrow\infty$, where $J_{T}$ is given by \eqref{eq9}.
\end{itemize}

Then we have the following verification theorem.
\begin{theorem}\label{theorem4.7}
Assume that the conditions (B.1), (B.10), (B.11) and (B.12) are satisfied. If $\omega\in C^{2}(\mathbb{R}^{+}\times\mathbb{R})$ is a solution of the HJB equation \eqref{eq24} and
\begin{align*}
(c^{*},\psi^{*})=\arg\max\limits_{c\in(0,\infty),\psi\in\mathbb{R}} \;&\left\{\left[r+\psi\eta(\nu)\right]x\omega_{x}-c\omega_{x}+\frac{1}{2}\psi^{2}x^{2}\omega_{xx}+m_{1}(\nu)\omega_{\nu}+\frac{1}{2}m^{2}_{2}(\nu)\omega_{\nu\nu}\right.\\
&\left.\quad+x\psi\rho m_{2}(\nu)\omega_{x\nu}+\beta(1-\frac{1}{\phi})^{-1}(1-\gamma)\omega\left[\left(\frac{c}{((1-\gamma)\omega)^{\frac{1}{1-\gamma}}}\right)^{1-\frac{1}{\phi}}-1\right]\right\},
\end{align*}
then $c^{*}$ and $\psi^{*}$ are optimal and $\omega$ coincides with the value function of the problem \eqref{eq15} when the aggregator $f(c,J)$ takes the form \eqref{eq10}.
\end{theorem}

If we differentiate the expression inside the bracket of \eqref{eq24} with respect to $c$ and $\psi$, respectively, then we can derive the optimal solution by
\begin{equation}\label{eq25}
   \left\{ \begin{aligned}
   &c^{*}_{t}=\left[\frac{\omega_{x}}{\beta(1-\gamma)\omega}((1-\gamma)\omega)^{\frac{1-\frac{1}{\phi}}{1-\gamma}}\right]^{-\phi},\\
   &\psi^{*}_{t}=-\frac{\eta(\nu)\omega_{x}+\rho m_{2}(\nu)\omega_{x\nu}}{x\omega_{xx}}.
  \end{aligned}\right.
\end{equation}
We make an ansatz that $\omega(x,\nu)=\frac{x^{1-\gamma}}{1-\gamma}h(\nu)^{-\frac{1-\gamma}{1-\phi}}$ for some deterministic function $h(\nu)$. By substituting $\eqref{eq25}$ and the above ansatz for $\omega(x,\nu)$ into $\eqref{eq24}$, we can derive that
\begin{align}\label{eq26}
  r&-\beta^{\phi}h^{-1}+\frac{1}{2\gamma}\eta^{2}(\nu)+\frac{\beta\phi}{\phi-1}\left(\beta^{\phi-1}h^{-1}-1\right)-\frac{1}{1-\phi}\left[m_{1}(\nu)+\frac{\eta(\nu)}{\gamma}\rho m_{2}(\nu)(1-\gamma)\right]\frac{h_{\nu}}{h}\nonumber\\
  &+\frac{ 1}{2(1-\phi)^{2}}m^{2}_{2}(\nu)\left[2-\phi-\gamma+\frac{\rho^{2}(1-\gamma)^{2}}{\gamma}\right]\frac{h^{2}_{\nu}}{h^{2}}-\frac{ 1}{2(1-\phi)}m^{2}_{2}(\nu)\frac{h_{\nu\nu}}{h}=0.
\end{align}
Since the equation \eqref{eq26} involves the term $\beta^{\phi}h^{-1}$, it usually does not admit an exact analytical solution. Therefore, we turn to find an approximate solution of \eqref{eq26}. It is well known that by linearizing the term $\beta^{\phi}h^{-1}$, the log-linear approximation method proposed in \cite{Campbell1993, Campbell2004} has been widely adopted to the study of dynamic investment/consumption problems \cite{Campbell1993, Campbell1999, Chacko2005, Hsuku2007, Liu2010} and its accuracy and convergence have been investigated in \cite{Campbell1993, Campbell2004}. In the following, we will employ the log-linear approximation method used in \cite{Campbell2004} to obtain an approximate solution of \eqref{eq26} by approximating the consumption-to-wealth ratio around its long-horizon mean. Inserting the conjecture of $\omega(x,\nu)$ into \eqref{eq25} yields
$
c^{*}_{t}=\beta^{\phi}h^{-1}x,
$
which shows that
$$
\beta^{\phi}h^{-1}=\frac{c^{*}_{t}}{x}=\exp\{\ln c^{*}_{t}-\ln x\}:=\exp\{\overline{c}^{*}_{t}-\overline{x}\},
$$
where $\overline{c}^{*}_{t}=\ln c^{*}_{t}$ and $\overline{x}=\ln x$. By the first-order Taylor expansion of $\beta^{\phi}h^{-1}$ around the expected value of the log consumption-wealth ratio $\mathbb{E}(\overline{c}^{*}_{t}-\overline{x})$, one has
\begin{align}\label{eq27}
   \beta^{\phi}h^{-1}\approx&\exp\{\mathbb{E}(\overline{c}^{*}_{t}-\overline{x})\}
   +\exp\{\mathbb{E}(\overline{c}^{*}_{t}-\overline{x})\}\left[\overline{c}^{*}_{t}-\overline{x}-\mathbb{E}(\overline{c}^{*}_{t}-\overline{x})\right]
   =\zeta_{1}+\zeta_{2}(\overline{c}^{*}_{t}-\overline{x}).
\end{align}
Here, $\zeta_{1}=\exp\{\mathbb{E}(\overline{c}^{*}_{t}-\overline{x})\}\left[1-\mathbb{E}(\overline{c}^{*}_{t}-\overline{x})\right]$ and $\zeta_{2}=\exp\{\mathbb{E}(\overline{c}^{*}_{t}-\overline{x})\}$. Moreover, one can obtain that
\begin{align}\label{eq28}
\beta^{\phi}h^{-1}\approx\zeta_{1}+\zeta_{2}(\ln c^{*}_{t}-\ln x)=\zeta_{1}+\zeta_{2}(\phi \ln\beta-\ln h).
\end{align}

By making a general ansatz for $h$ as an exponential-polynomial form, we will show that when the order of the polynomial is higher than 2, the problem \eqref{eq15} with $f(c,J)$ given by \eqref{eq10} is approximately unsolvable under the log-linear approximation method.
\begin{theorem}\label{theorem4.2}
If the condition (B.7) is satisfied and $h(\nu)$ is conjectured in the exponential-polynomial form such that
\begin{align}\label{eq29}
\omega(x,\nu)=\frac{x^{1-\gamma}}{1-\gamma}h(\nu)^{-\frac{1-\gamma}{1-\phi}}=\frac{x^{1-\gamma}}{1-\gamma}\exp\left\{-\frac{1-\gamma}{1-\phi}
\left(A_{0}+\sum\limits_{k=1}^{n}\frac{1}{k}A_{k}\nu^{k}\right)\right\},
\end{align}
where $A_{k}$ for all $k=0,\cdot\cdot\cdot,n$ are constants, then the PDE \eqref{eq26} has no approximate solutions under the log-linear approximation method when the order $n$ of the polynomial is higher than 2.
\end{theorem}
\begin{proof}
See Appendix A.
\end{proof}

In addition, we will prove that by the form of \eqref{eq29}, the problem \eqref{eq15} with $f(c,J)$ given by \eqref{eq10} can be approximately solvable under the log-linear approximation method when $n=0,1,2$.

\begin{theorem}\label{theorem4.8}
Assume that $h$ is conjectured in the form of \eqref{eq29}. (i) If the condition (B.8) is satisfied, then the PDE \eqref{eq26} has an approximate solution under the log-linear approximation method when $n=0$; (ii) If the condition (B.9) is satisfied, then the PDE \eqref{eq26} has an approximate solution under the log-linear approximation method when $n=1$; (iii) If the condition (B.7) is satisfied, then the PDE \eqref{eq26} has an approximate solution under the log-linear approximation method when $n=2$.
\end{theorem}
\begin{proof}
See Appendix B.
\end{proof}

\subsection{The case of finite time-horizon}
Now, we will concentrate on the case of the finite time interval. If the investor has unit EIS, which means the aggregator $f(c,J)$ takes the form \eqref{eq11} in the problem \eqref{eq17}, we assume that $c\in[C_{m},C_{M}]$ and then follow \cite{Kraft2017, Kraft2013} to derive the corresponding HJB equation
\begin{equation}\label{eq32}
\begin{aligned}
\sup\limits_{c\in[C_{m},C_{M}],\psi\in\mathbb{R}} \;&\left\{\omega_{t}+\left[r+\psi\eta(\nu)\right]x\omega_{x}-c\omega_{x}+\frac{1}{2}\psi^{2}x^{2}\omega_{xx}+m_{1}(\nu)\omega_{\nu}+\frac{1}{2}m^{2}_{2}(\nu)\omega_{\nu\nu}\right.\\
&\left.\quad+x\psi\rho m_{2}(\nu)\omega_{x\nu}+\beta(1-\gamma)\omega\left[\ln c-\frac{1}{1-\gamma}\ln ((1-\gamma)\omega)\right]\right\}=0
\end{aligned}
\end{equation}
with the boundary condition $\omega(T,x,\nu)=\varepsilon^{1-\gamma}\frac{x^{1-\gamma}}{1-\gamma}$. To make sure that the candidate value function $\omega$ in \eqref{eq32} is indeed the value function of the problem \eqref{eq17} when the aggregator $f(c,J)$ takes the form \eqref{eq11}, we assume that
\begin{itemize}
  \item[(B.13)]
      $\int^{\cdot}_{t}\omega_{x}(X_{s})\psi_{s}X_{s}\sqrt{1-\rho^{2}}\mathrm{d}W^{\perp}_{s}$ and $\int^{\cdot}_{t}\left[\omega_{x}(X_{s})\psi_{s}X_{s}\rho+\omega_{\nu}(X_{s})m_{2}(\nu_{s})\right]\mathrm{d}W^{\nu}_{s}$ are true martingales for any $c\in[C_{m},C_{M}]$ and $\psi\in\mathbb{R}$.
\end{itemize}

Thus, we have the following result.
\begin{theorem}\label{theorem4.9}
Assume that the conditions (B.2) and (B.13) are satisfied. If $\omega\in C^{1,2}([0,T]\times\mathbb{R}^{+}\times\mathbb{R})$ is a solution of the HJB equation \eqref{eq32} and
\begin{align*}
(c^{*},\psi^{*})=\arg\max\limits_{c\in[C_{m},C_{M}],\psi\in\mathbb{R}} \;&\left\{\omega_{t}+\left[r+\psi\eta(\nu)\right]x\omega_{x}-c\omega_{x}+\frac{1}{2}\psi^{2}x^{2}\omega_{xx}+m_{1}(\nu)\omega_{\nu}+\frac{1}{2}m^{2}_{2}(\nu)\omega_{\nu\nu}\right.\\
&\left.\quad+x\psi\rho m_{2}(\nu)\omega_{x\nu}+\beta(1-\gamma)\omega\left[\ln c-\frac{1}{1-\gamma}\ln ((1-\gamma)\omega)\right]\right\},
\end{align*}
then $c^{*}$ and $\psi^{*}$ are optimal and $\omega$ coincides with the value function of the problem \eqref{eq17} when the aggregator $f(c,J)$ takes the form \eqref{eq11}.
\end{theorem}
\begin{proof}
The proof, similar to that of Theorems \ref{theorem4.5}, is omitted here.
\end{proof}

Solving the maximization problem \eqref{eq32} with respect to $c$ and $\psi$, we have the following optimal solution
\begin{equation}\label{eq33}
   \left\{ \begin{aligned}
   &c^{*}_{t}=\frac{\beta(1-\gamma)\omega}{\omega_{x}},\\
   &\psi^{*}_{t}=-\frac{\eta(\nu)\omega_{x}+\rho m_{2}(\nu)\omega_{x\nu}}{x\omega_{xx}}.
  \end{aligned}\right.
\end{equation}
Moreover, we conjecture a solution of the form $\omega(t,x,\nu)=\frac{x^{1-\gamma}}{1-\gamma}h(t,\nu)^{1-\gamma}$ for some deterministic function $h(t,\nu)$. Substituting $\eqref{eq33}$ and the ansatz for $\omega(t,x,\nu)$ into $\eqref{eq32}$ yields the following PDE
\begin{align}\label{eq34}
  h_{t}&+\left[r-\beta+\frac{1}{2\gamma}\eta^{2}(\nu)+\beta\left(\ln\beta-\ln h\right)\right]h+\left[m_{1}(\nu)+\frac{\eta(\nu)}{\gamma}\rho m_{2}(\nu)(1-\gamma)\right]h_{\nu}\nonumber\\
  &+\frac{ 1}{2}m^{2}_{2}(\nu)\left[\frac{\rho^{2}(1-\gamma)^{2}}{\gamma}-\gamma\right]\frac{h^{2}_{\nu}}{h}+\frac{ 1}{2}m^{2}_{2}(\nu)h_{\nu\nu}=0.
\end{align}

In the following, we will assume a general ansatz for $h$ as an exponential-polynomial form and show that when the order of the polynomial is higher than 2, the problem \eqref{eq17} with $f(c,J)$ given by \eqref{eq11} is unsolvable under this ansatz.
\begin{theorem}\label{theorem4.3}
If the condition (B.7) is satisfied and $h(t,\nu)$ is conjectured in the exponential-polynomial form such that
\begin{align}\label{eq35}
\omega(t,x,\nu)=\frac{x^{1-\gamma}}{1-\gamma}h(t,\nu)^{1-\gamma}=\frac{x^{1-\gamma}}{1-\gamma}\exp\left\{(1-\gamma)\left(A_{0}(T-t)+
\sum\limits_{k=1}^{n}\frac{1}{k}A_{k}(T-t)\nu^{k}\right)\right\},
\end{align}
where $A_{k}(\tau)$ for all $k=0,\cdot\cdot\cdot,n$ are functions of $\tau$ with $A_{0}(0)=\ln \varepsilon$ and $A_{k}(0)=0\;(k=1,\cdot\cdot\cdot,n)$, then the PDE \eqref{eq34} has no analytical solutions when the order $n$ of the polynomial is higher than 2.
\end{theorem}
\begin{proof}
See Appendix C.
\end{proof}

Furthermore, we will prove that the problem \eqref{eq17} with $f(c,J)$ given by \eqref{eq11} can be solvable under the form of \eqref{eq35} when $n=0,1,2$.

\begin{theorem}\label{theorem4.10}
Assume that $h$ is conjectured in the form of \eqref{eq35}. (i) If the condition (B.8) is satisfied, then the PDE \eqref{eq34} has an analytical solution when $n=0$; (ii) If the condition (B.9) is satisfied, then the PDE \eqref{eq34} has an analytical solution when $n=1$; (iii) If the condition (B.7) is satisfied, then the PDE \eqref{eq34} has an analytical solution when $n=2$.
\end{theorem}
\begin{proof}
See Appendix D.
\end{proof}

Then, we study the case of $\phi\neq1$. That is, the aggregator $f(c,J)$ takes the form \eqref{eq10} in the problem \eqref{eq17}. Following \cite{Kraft2017, Kraft2013}, the corresponding HJB equation is
\begin{equation}\label{eq37}
\begin{aligned}
\sup\limits_{c\in(0,\infty),\psi\in\mathbb{R}} \;&\left\{\omega_{t}+\left[r+\psi\eta(\nu)\right]x\omega_{x}-c\omega_{x}+\frac{1}{2}\psi^{2}x^{2}\omega_{xx}+m_{1}(\nu)\omega_{\nu}+\frac{1}{2}m^{2}_{2}(\nu)\omega_{\nu\nu}\right.\\
&\left.\quad+x\psi\rho m_{2}(\nu)\omega_{x\nu}+\beta(1-\frac{1}{\phi})^{-1}(1-\gamma)\omega\left[\left(\frac{c}{((1-\gamma)\omega)^{\frac{1}{1-\gamma}}}\right)^{1-\frac{1}{\phi}}-1\right]\right\}=0
\end{aligned}
\end{equation}
with the boundary condition $\omega(T,x,\nu)=\varepsilon^{1-\gamma}\frac{x^{1-\gamma}}{1-\gamma}$. For applying Theorem 3.1 in \cite{Kraft2013} to ensure that the candidate value function $\omega$ in \eqref{eq37} is indeed the value function of the problem \eqref{eq17} when the aggregator $f(c,J)$ takes the form \eqref{eq10}, we assume that
\begin{itemize}
  \item[(B.14)] for any $c\in(0,\infty)$ and $\psi\in\mathbb{R}$,
      $\int^{\cdot}_{t}\omega_{x}(X_{s})\psi_{s}X_{s}\sqrt{1-\rho^{2}}\mathrm{d}W^{\perp}_{s}$ and $\int^{\cdot}_{t}\left[\omega_{x}(X_{s})\psi_{s}X_{s}\rho+\omega_{\nu}(X_{s})m_{2}(\nu_{s})\right]\mathrm{d}W^{\nu}_{s}$ are true martingales
\end{itemize}
and then give the following verification theorem.
\begin{theorem}\label{theorem4.11}
Assume that the conditions (B.2), (B.10) and (B.14) are satisfied. If $\omega\in C^{1,2}([0,T]\times\mathbb{R}^{+}\times\mathbb{R})$ is a solution of the HJB equation \eqref{eq37} and
\begin{align*}
(c^{*},\psi^{*})=\arg\max\limits_{c\in(0,\infty),\psi\in\mathbb{R}} \;&\left\{\omega_{t}+\left[r+\psi\eta(\nu)\right]x\omega_{x}-c\omega_{x}+\frac{1}{2}\psi^{2}x^{2}\omega_{xx}+m_{1}(\nu)\omega_{\nu}+\frac{1}{2}m^{2}_{2}(\nu)\omega_{\nu\nu}\right.\\
&\left.\quad+x\psi\rho m_{2}(\nu)\omega_{x\nu}+\beta(1-\frac{1}{\phi})^{-1}(1-\gamma)\omega\left[\left(\frac{c}{((1-\gamma)\omega)^{\frac{1}{1-\gamma}}}\right)^{1-\frac{1}{\phi}}-1\right]\right\},
\end{align*}
then $c^{*}$ and $\psi^{*}$ are optimal and $\omega$ coincides with the value function of the problem \eqref{eq17} when the aggregator $f(c,J)$ takes the form \eqref{eq10}.
\end{theorem}

By Theorem \ref{theorem4.11}, we can derive that the optimal solution follows
\begin{equation}\label{eq38}
   \left\{ \begin{aligned}
   &c^{*}_{t}=\left[\frac{\omega_{x}}{\beta(1-\gamma)\omega}((1-\gamma)\omega)^{\frac{1-\frac{1}{\phi}}{1-\gamma}}\right]^{-\phi},\\
   &\psi^{*}_{t}=-\frac{\eta(\nu)\omega_{x}+\rho m_{2}(\nu)\omega_{x\nu}}{x\omega_{xx}}.
  \end{aligned}\right.
\end{equation}
We assume $\omega(t,x,\nu)=\frac{x^{1-\gamma}}{1-\gamma}h(t,\nu)^{-\frac{1-\gamma}{1-\phi}}$ for some deterministic function $h(t,\nu)$. Substituting $\eqref{eq38}$ and the above form of $\omega(t,x,\nu)$ into \eqref{eq37} leads to
\begin{align}\label{eq39}
  -\frac{h_{t}}{(1-\phi)h}&+r-\beta^{\phi}h^{-1}+\frac{1}{2\gamma}\eta^{2}(\nu)+\frac{\beta\phi}{\phi-1}\left(\beta^{\phi-1}h^{-1}-1\right)-\frac{1}{1-\phi}\left[m_{1}(\nu)+\frac{\eta(\nu)}{\gamma}\rho m_{2}(\nu)(1-\gamma)\right]\frac{h_{\nu}}{h}\nonumber\\
  &+\frac{ 1}{2(1-\phi)^{2}}m^{2}_{2}(\nu)\left[2-\phi-\gamma+\frac{\rho^{2}(1-\gamma)^{2}}{\gamma}\right]\frac{h^{2}_{\nu}}{h^{2}}-\frac{ 1}{2(1-\phi)}m^{2}_{2}(\nu)\frac{h_{\nu\nu}}{h}=0.
\end{align}
In addition, we have $c^{*}_{t}=\beta^{\phi}h^{-1}x$.

Due to the nonlinear term $\beta^{\phi}h^{-1}$, \eqref{eq39} does not admit an exact analytical solution in general. Thus, we attempt to find an approximate solution of \eqref{eq39}. When adopting the log-linear approximation method within the infinite horizon, the consumption-to-wealth ratio can be written as a linear function of the log consumption-to-wealth ratio and the coefficients of the linear function are constants (see, for example, the equation (B.5) of \cite{Campbell2004} or \eqref{eq28}). However, when the horizon is finite, the coefficients depend on time to the horizon, which makes the approximation solution less tractable. For consumption and investment problems under recursive utility and finite horizon, Campani and Garcia \cite{Campani2019} put forward a variation of the log-linear approach by approximating the consumption-to-wealth ratio around the time-varying consumption strategy obtained for a fixed point of the state-variable space. This fixed point can be chosen by investors conveniently based on current market conditions and expectations. Moreover, they discussed the accuracy of the derived approximation solution by numerical experiments. In the sequel, we will follow \cite{Campani2019} to adopt the log-linear approach by approximating the consumption-to-wealth ratio around the long-run mean of the state variable $\widetilde{v}$ under finite horizon. Specially, by the first-order Taylor expansion of $\beta^{\phi}h^{-1}$ around the time-varying log consumption-to-wealth strategy obtained for the long-run mean of the state variable $(\overline{c}^{*}_{t}-\overline{x})|_{v_{t}=\widetilde{v}}$, one has
\begin{align}\label{eq43}
   \beta^{\phi}h^{-1}\approx&\exp\{(\overline{c}^{*}_{t}-\overline{x})|_{v_{t}=\widetilde{v}}\}
   +\exp\{(\overline{c}^{*}_{t}-\overline{x})|_{v_{t}=\widetilde{v}}\}\left[\overline{c}^{*}_{t}-\overline{x}-(\overline{c}^{*}_{t}-\overline{x})|_{v_{t}=\widetilde{v}}\right]
   =\zeta_{3}+\zeta_{4}(\overline{c}^{*}_{t}-\overline{x}),
\end{align}
where $\zeta_{3}=\exp\{(\overline{c}^{*}_{t}-\overline{x})|_{v_{t}=\widetilde{v}}\}\left[1-(\overline{c}^{*}_{t}-\overline{x})|_{v_{t}=\widetilde{v}}\right]$ and $\zeta_{4}=\exp\{(\overline{c}^{*}_{t}-\overline{x})|_{v_{t}=\widetilde{v}}\}$. Furthermore, one can obtain that
\begin{align}\label{eq44}
\beta^{\phi}h^{-1}\approx\zeta_{3}+\zeta_{4}(\ln c^{*}_{t}-\ln x)=\zeta_{3}+\zeta_{4}(\phi \ln\beta-\ln h).
\end{align}
Therefore, it follows from \eqref{eq44} that \eqref{eq39} can be approximated by the following formula
\begin{align}\label{eq40}
  -&\frac{h_{t}}{(1-\phi)h}+r-\zeta_{3}-\zeta_{4}(\phi \ln\beta-\ln h)+\frac{1}{2\gamma}\eta^{2}(\nu)\nonumber\\
  &+\frac{\beta\phi}{\phi-1}\left(\beta^{-1}(\zeta_{3}+\zeta_{4}(\phi \ln\beta-\ln h))-1\right)-\frac{1}{1-\phi}\left[m_{1}(\nu)+\frac{\eta(\nu)}{\gamma}\rho m_{2}(\nu)(1-\gamma)\right]\frac{h_{\nu}}{h}\nonumber\\
  &+\frac{ 1}{2(1-\phi)^{2}}m^{2}_{2}(\nu)\left[2-\phi-\gamma+\frac{\rho^{2}(1-\gamma)^{2}}{\gamma}\right]\frac{h^{2}_{\nu}}{h^{2}}-\frac{ 1}{2(1-\phi)}m^{2}_{2}(\nu)\frac{h_{\nu\nu}}{h}=0.
\end{align}
By assuming that $h$ takes a form of exponential-polynomial, we will discuss that when the order of the polynomial is higher than 2, the PDE \eqref{eq40} is unsolvable and thus the problem \eqref{eq17} with $f(c,J)$ given by \eqref{eq10} is approximately unsolvable under the log-linear approximation method.
\begin{theorem}\label{theorem4.4}
If the condition (B.7) is satisfied and $h(t,\nu)$ is conjectured in the exponential-polynomial form such that
\begin{align}\label{eq41}
\omega(t,x,\nu)=\frac{x^{1-\gamma}}{1-\gamma}h(t,\nu)^{-\frac{1-\gamma}{1-\phi}}=\frac{x^{1-\gamma}}{1-\gamma}\exp\left\{-\frac{1-\gamma}{1-\phi}\left(A_{0}(T-t)+
\sum\limits_{k=1}^{n}\frac{1}{k}A_{k}(T-t)\nu^{k}\right)\right\},
\end{align}
where $A_{k}(\tau)$ for all $k=0,\cdot\cdot\cdot,n$ are functions of $\tau$ with $A_{0}(0)=(\phi-1)\ln \varepsilon$ and $A_{k}(0)=0\;(k=1,\cdot\cdot\cdot,n)$, then the PDE \eqref{eq40} has no analytical solutions when the order $n$ of the polynomial is higher than 2, that is to say, the PDE \eqref{eq39} has no approximate solutions under the log-linear approximation method when $n>2$.
\end{theorem}
\begin{proof}
See Appendix E.
\end{proof}

\begin{remark}
 We note that Liu \cite{Liu2007} considered the portfolio and consumption optimization problems with the power utility in finite time region. Chen and Escobar-Anel \cite{Cheng2023} pointed out that "Liu \cite{Liu2007} was the first to mention that a polynomial with an order higher than 2 cannot help in solving the PDE, but no proof was provided" and thus they presented the proof of the solvability of the portfolio optimization problem with the power utility in finite time region under the conjecture of exponential-polynomial of the value function. However, to the best of our knowledge, when considering the portfolio and consumption optimization problems, how to give the proof that a polynomial with an order higher than 2 cannot help in solving the PDE is still in a blank. Fortunately, due to the fact that the Epstein-Zin aggregator can reduce to the power utility, the study of Theorem \ref{theorem4.4} can be applied to fill in such a blank.
\end{remark}

We will further verify that by the form of \eqref{eq41}, the problem \eqref{eq17} with $f(c,J)$ given by \eqref{eq10} can be approximately solvable under the log-linear approximation method when $n=0,1,2$.

\begin{theorem}\label{theorem4.12}
Assume that $h$ is conjectured in the form of \eqref{eq41}. (i) If the condition (B.8) is satisfied, then the PDE \eqref{eq39} has an approximate solution under the log-linear approximation method when $n=0$; (ii) If the condition (B.9) is satisfied, then the PDE \eqref{eq39} has an approximate solution under the log-linear approximation method when $n=1$; (iii) If the condition (B.7) is satisfied, then the PDE \eqref{eq39} has an approximate solution under the log-linear approximation method when $n=2$.
\end{theorem}
\begin{proof}
See Appendix F.
\end{proof}

\section{Applications}
In this section, we will concentrate on a class of famous stochastic volatility model in which the volatility is captured by a square-root form as suggested by Heston \cite{heston1993closed}. Specially, the dynamics of the risky asset follows
\begin{equation}
   \left\{ \begin{aligned}\label{eq45}
   &\mathrm{d}S_{t}=S_{t}\left[(r+\xi \nu_{t})\mathrm{d}t+\sqrt{\nu_{t}}(\rho\mathrm{d}W^{\nu}_{t}+\sqrt{1-\rho^{2}}\mathrm{d}W^{\perp}_{t})\right], \\
   &\mathrm{d}\nu_{t}=\kappa(\theta-\nu_{t})\mathrm{d}t+\sigma\sqrt{\nu_{t}}\mathrm{d}W^{\nu}_{t},
  \end{aligned}\right.
\end{equation}
where $r$, $\xi$, $\kappa$, $\theta$ and $\sigma$ are positive constants representing the risk-free interest rate, the premium for volatility, the mean-reversion rate, the long-run mean and the volatility of volatility, respectively, $\rho$ denotes the correlation coefficient between $S_{t}$ and $\nu_{t}$. It follows from equations \eqref{eq2}, \eqref{eq3} and \eqref{eq45} that $\eta(\nu_{t})=\xi \sqrt{\nu_{t}}$, $G(\nu_{t}, t)=\sqrt{\nu_{t}}$, $m_{1}(\nu_{t})=\kappa(\theta-\nu_{t})$ and $m_{2}(\nu_{t})=\sigma\sqrt{\nu_{t}}$. Then in this case the equation \eqref{eq14} becomes
\begin{align}\label{eq46}
\mathrm{d}X_{t}=\left[r+\xi \sqrt{\nu_{t}}\psi_{t}\right]X_{t}\mathrm{d}t-c_{t}\mathrm{d}t+\psi_{t}X_{t}(\rho\mathrm{d}W^{\nu}_{t}+\sqrt{1-\rho^{2}}\mathrm{d}W^{\perp}_{t}).
\end{align}
Moreover, we focus on the case of $\phi\leq1$ and $\gamma>1$.

Since $\eta^{2}(\nu)$, $m_{1}(\nu)$, $\eta(\nu) m_{2}(\nu)$ and $m^{2}_{2}(\nu)$ are linear in $\nu$, $h$ is conjectured in the exponential-linear form in Theorems \ref{theorem4.1}, \ref{theorem4.2}, \ref{theorem4.3} and \ref{theorem4.4}. In addition, we can adopt the results in Section 4 to give explicit solutions to the consumption-portfolio problems with recursive preferences under Heston's model as follows.
\begin{proposition}\label{proposition5.1}
{\rm(i)} Assume that the investor's objective is captured by \eqref{eq15} where $f(c,J)$ is given by \eqref{eq11}. If the conditions (B.1), (B.5) and (B.6) are satisfied and $
(1-2\gamma)\rho^{2}\sigma^{2}(1-\gamma)^{2}\hat{A}_{1}^{2}+2\xi\rho\sigma(1-\gamma)^{2}\hat{A}_{1}+\xi^{2}\leq\frac{\kappa^{2}\gamma^{2}}{2(1-\gamma)\sigma^{2}}
$, then the optimal consumption-investment strategy follows
\begin{equation}\label{eq47}
   \left\{ \begin{aligned}
   &\hat{c}^{*}_{t}=\beta x,\\
   &\hat{\pi}^{*}_{t}=\frac{\xi}{\gamma}+\frac{\rho\sigma(1-\gamma)\hat{A}_{1}}{\gamma}
  \end{aligned}\right.
\end{equation}
and the value function is
\begin{equation}\label{eq48}
  \hat{\omega}(x,\nu)=\frac{x^{1-\gamma}}{1-\gamma}\exp\{(1-\gamma)(\hat{A}_{0}+\hat{A}_{1}\nu)\},
\end{equation}
where $\hat{A}_{0}$ and $\hat{A}_{1}$ are given by equations \eqref{eq49} and \eqref{eq50}, respectively.

{\rm(ii)} Assume that the investor's objective is described by \eqref{eq15} where $f(c,J)$ follows \eqref{eq10}. If the conditions (B.1), (B.10), (B.11) and (B.12) are satisfied and $
(1-2\gamma)\rho^{2}\sigma^{2}(1-\gamma)^{2}\frac{\check{A}_{1}^{2}}{(1-\phi)^{2}}-2\xi\rho\sigma(1-\gamma)^{2}\frac{\check{A}_{1}}{1-\phi}+\xi^{2}\leq\frac{\kappa^{2}\gamma^{2}}{2(1-\gamma)\sigma^{2}}
$, then the optimal approximate consumption-investment strategy satisfies
\begin{equation}\label{eq54}
   \left\{ \begin{aligned}
   &\check{c}^{*}_{t}=\beta^{\phi}\check{h}^{-1}x,\\
   &\check{\pi}^{*}_{t}=\frac{\xi}{\gamma}-\frac{\rho\sigma(1-\gamma)\frac{\check{A}_{1}}{1-\phi}}{\gamma}
  \end{aligned}\right.
\end{equation}
and the value function is
\begin{equation}\label{eq55}
  \check{\omega}(x,\nu)=\frac{x^{1-\gamma}}{1-\gamma}\check{h}(\nu)^{-\frac{1-\gamma}{1-\phi}}
\end{equation}
with $\check{h}(\nu)=\exp\{\check{A}_{0}+\check{A}_{1}\nu\}$, where $\check{A}_{0}$ and $\check{A}_{1}$ are given by equations \eqref{eq52} and \eqref{eq53}, respectively.

{\rm(iii)} Assume that the investor's objective is given by \eqref{eq17} where $f(c,J)$ is defined by \eqref{eq11}. If the conditions (B.2) and (B.13) are satisfied, $(1-2\gamma)\rho^{2}\sigma^{2}(1-\gamma)^{2}\tilde{A}_{1}^{2}(t)+2\xi\rho\sigma(1-\gamma)^{2}\tilde{A}_{1}(t)+\xi^{2}\leq\frac{\kappa^{2}\gamma^{2}}{2(1-\gamma)\sigma^{2}}$ and $\tilde{A}_{1}(t)>0$ for $t\in[0,T]$, then the optimal consumption-investment strategy
can be written as
\begin{equation}\label{eq58}
   \left\{ \begin{aligned}
   &\tilde{c}^{*}_{t}=\beta x,\\
   &\tilde{\pi}^{*}_{t}=\frac{\xi}{\gamma}+\frac{\rho\sigma(1-\gamma)\tilde{A}_{1}(T-t)}{\gamma}
  \end{aligned}\right.
\end{equation}
and the value function is
\begin{equation}\label{eq59}
  \tilde{\omega}(t,x,\nu)=\frac{x^{1-\gamma}}{1-\gamma}\exp\{(1-\gamma)(\tilde{A}_{0}(T-t)+\tilde{A}_{1}(T-t)\nu)\},
\end{equation}
where $\tilde{A}_{0}$ and $\tilde{A}_{1}$ are given by equations \eqref{eq56} and \eqref{eq57}, respectively.

{\rm(iv)} Assume that the investor's objective is depicted by \eqref{eq17} where $f(c,J)$ satisfies \eqref{eq10}. If the conditions (B.2), (B.10) and (B.14) are satisfied, $
(1-2\gamma)\rho^{2}\sigma^{2}(1-\gamma)^{2}\frac{\underline{A}_{1}^{2}(t)}{(1-\phi)^{2}}-2\xi\rho\sigma(1-\gamma)^{2}\frac{\underline{A}_{1}(t)}{1-\phi}+\xi^{2}\leq\frac{\kappa^{2}\gamma^{2}}{2(1-\gamma)\sigma^{2}}
$ and $\underline{A}_{1}(t)<0$ for $t\in[0,T]$, then the optimal approximate consumption-investment strategy is
\begin{equation}\label{eq66}
   \left\{ \begin{aligned}
   &\underline{c}^{*}_{t}=\beta^{\phi}\underline{h}^{-1}x,\\
   &\underline{\pi}^{*}_{t}=\frac{\xi}{\gamma}-\frac{\rho\sigma(1-\gamma)\frac{\underline{A}_{1}(T-t)}{1-\phi}}{\gamma}
  \end{aligned}\right.
\end{equation}
and the value function is
\begin{equation}\label{eq67}
  \underline{\omega}(t,x,\nu)=\frac{x^{1-\gamma}}{1-\gamma}\underline{h}(t,\nu)^{-\frac{1-\gamma}{1-\phi}}
\end{equation}
with $\underline{h}(t,\nu)=\exp\{\underline{A}_{0}(T-t)+\underline{A}_{1}(T-t)\nu\}$, where $\underline{A}_{0}$ and $\underline{A}_{1}$ are given by equations \eqref{eq65} and \eqref{eq63}, respectively.
\end{proposition}
\begin{proof}
See Appendix G.
\end{proof}

\begin{remark}\label{remark3}
For the consumption and investment problem with recursive utility and finite horizon under Heston's model, Kraft at al. \cite{Kraft2013} chosen the particular configuration of preference parameters that cancel out the nonlinear terms in the Bellman equation and thus derived an analytical solution for the case of $\phi\neq1$, with the unfortunate consequence of tying again the parameter of risk aversion to the EIS. However, Proposition \ref{proposition5.1} (iv) considers an approximate solution where the parameter of risk aversion and the EIS are separate, which is different from the case of \cite{Kraft2013}.
\end{remark}

Next, we will carry out numerical experiments to capture the behavior of the optimal consumption and investment strategies. In our calculations, we set $\varepsilon=1$, and other model parameters are taken from \cite{Kraft2013}: $\beta=0.08$, $\gamma=2$, $\phi=0.125$, $r=0.05$, $\xi=\frac{7}{15}$, $\rho=-0.5$, $\kappa=5$, $\theta=0.0225$, $\sigma=0.25$ and $T=10$ (years).

Figures \ref{fig1}-\ref{fig4} demonstrate the results of Proposition \ref{proposition5.1} for the above selected parameters. Also we provide the optimal consumption and investment strategies for the investor with different risk-averse attitudes.

\begin{figure}[htbp]
	\begin{minipage}[t]{0.5\linewidth}
		\centering
		\includegraphics[width=3.0in]{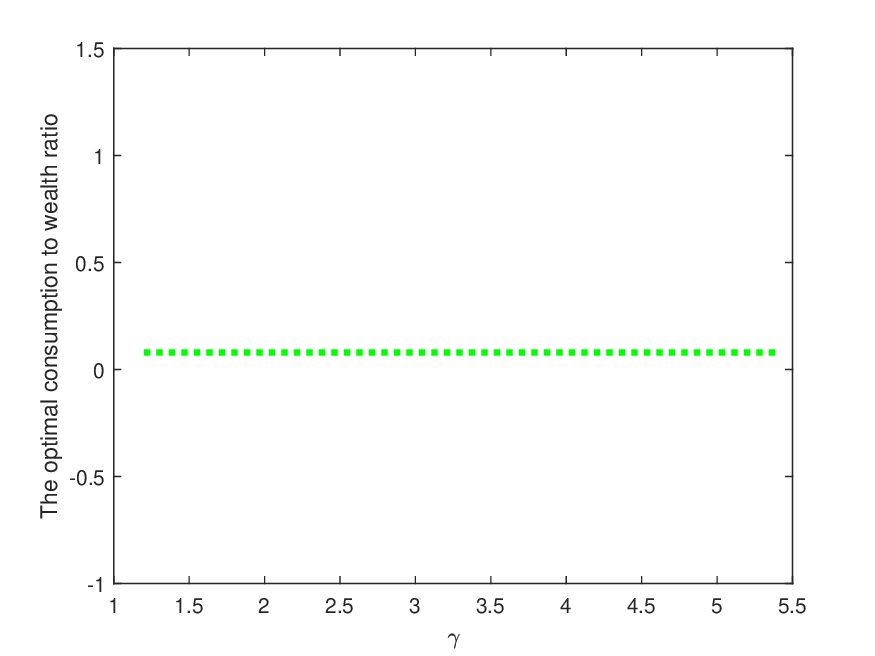}
        \caption*{(a)\;the effect of $\gamma$ on $\frac{\hat{c}^{*}_{t}}{X_{t}}$}
	\end{minipage}%
	\begin{minipage}[t]{0.5\linewidth}
		\centering
		\includegraphics[width=3.0in]{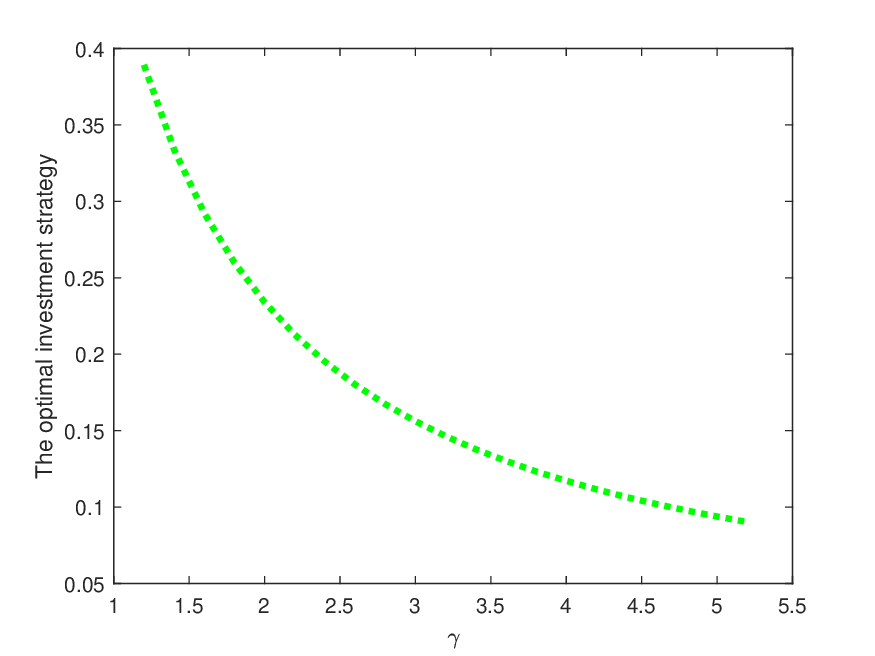}
        \caption*{(b)\;the effect of $\gamma$ on $\hat{\pi}^{*}_{t}$}
	\end{minipage}
    \caption{The effect of $\gamma$ on the optimal consumption-investment strategy \eqref{eq47}.}
	\label{fig1}
\end{figure}

\begin{figure}[htbp]
	\begin{minipage}[t]{0.5\linewidth}
		\centering
		\includegraphics[width=3.0in]{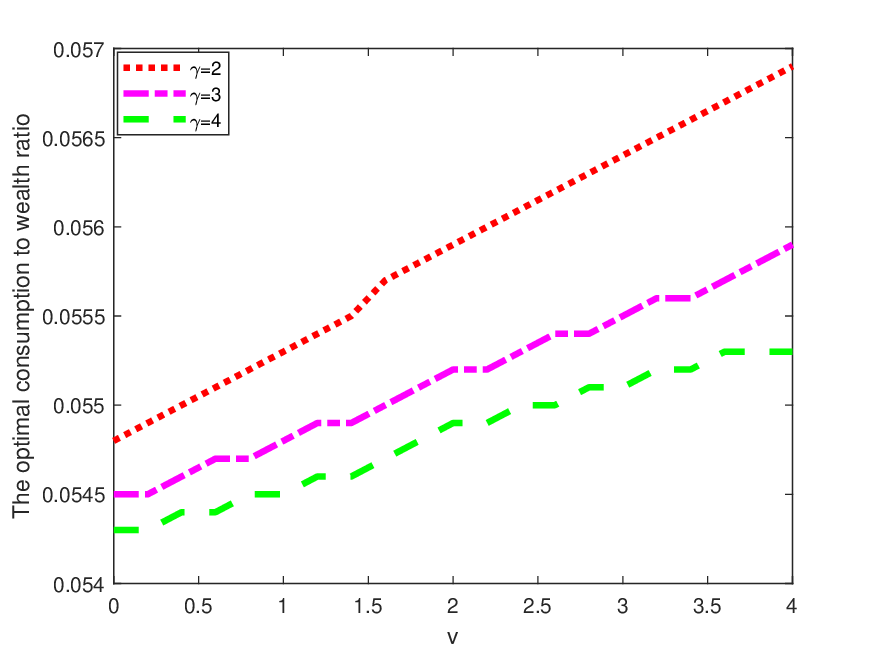}
        \caption*{(a)\;the effect of $\gamma$ on $\frac{\check{c}^{*}_{0}}{X_{0}}$}
	\end{minipage}%
	\begin{minipage}[t]{0.5\linewidth}
		\centering
		\includegraphics[width=3.0in]{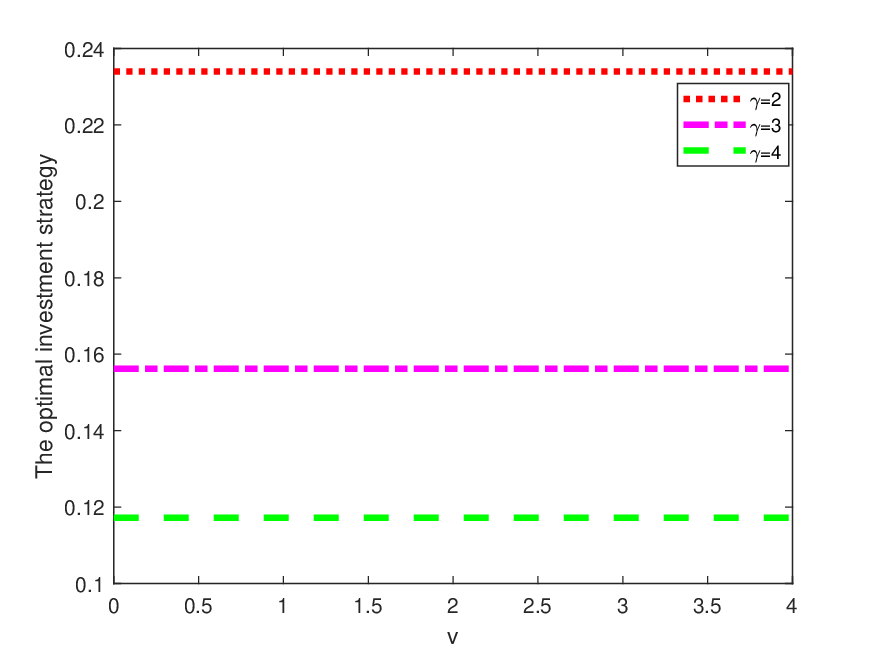}
        \caption*{(b)\;the effect of $\gamma$ on $\check{\pi}^{*}_{0}$}
	\end{minipage}
    \caption{The effect of $\gamma$ on the optimal consumption-investment strategy \eqref{eq54} at time $t=0$.}
	\label{fig2}
\end{figure}

\begin{figure}[htbp]
	\begin{minipage}[t]{0.5\linewidth}
		\centering
		\includegraphics[width=3.0in]{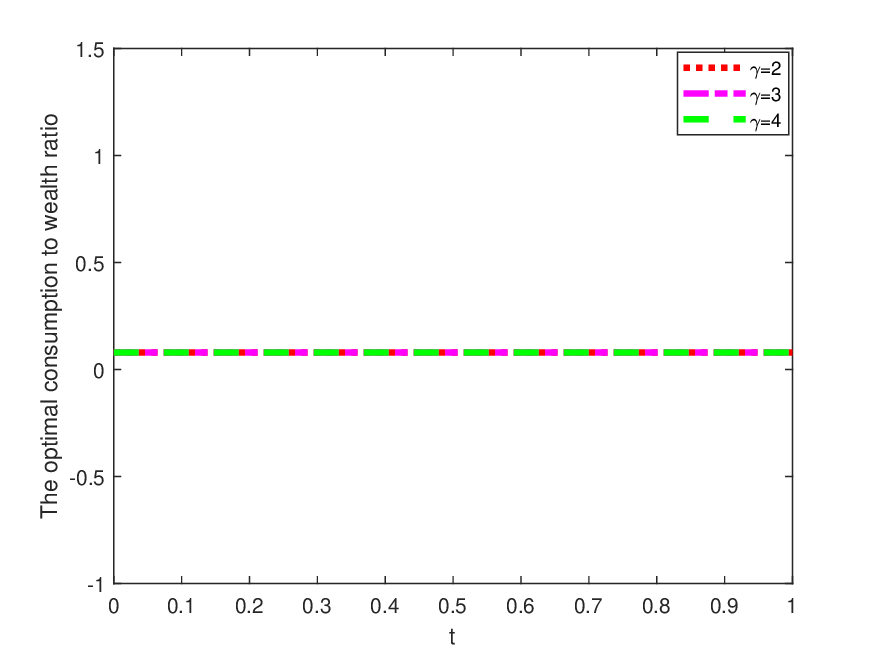}
        \caption*{(a)\;the effect of $\gamma$ on $\frac{\tilde{c}^{*}_{t}}{X_{t}}$}
	\end{minipage}%
	\begin{minipage}[t]{0.5\linewidth}
		\centering
		\includegraphics[width=3.0in]{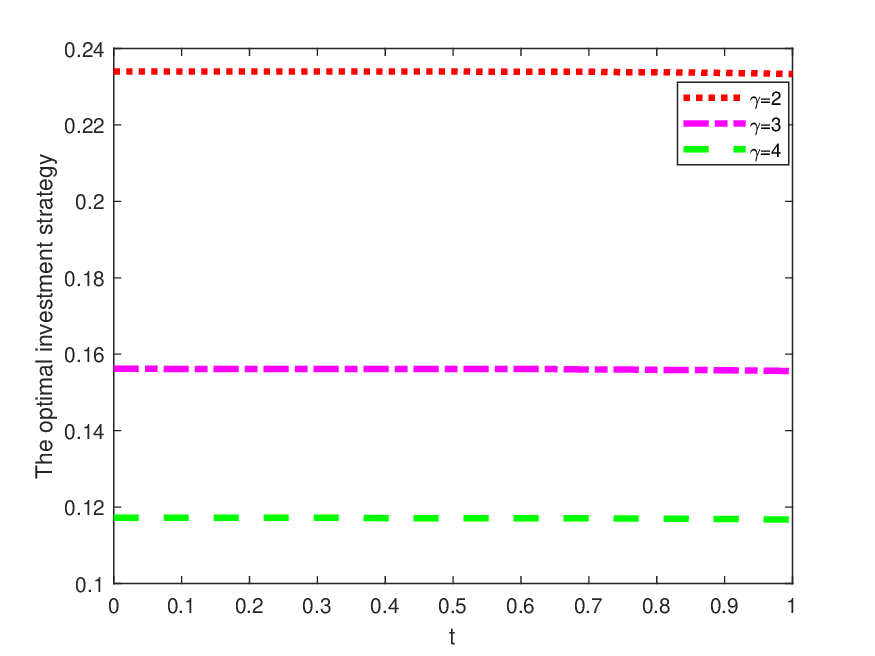}
        \caption*{(b)\;the effect of $\gamma$ on $\tilde{\pi}^{*}_{t}$}
	\end{minipage}
    \caption{The effect of $\gamma$ on the optimal consumption-investment strategy \eqref{eq58}.}
	\label{fig3}
\end{figure}

Figure \ref{fig1} (a) and Figure \ref{fig3} (a) show that the optimal consumption-to-wealth ratio is kept constant and equals to $\beta$ no matter how $\gamma$ is taken, which is consistent with the expressions of \eqref{eq47} and \eqref{eq58}. Figure \ref{fig2} (a) and Figure \ref{fig4} (a) display the optimal consumption-to-wealth ratio as a function of volatility at time $t=0$. Specially, the optimal consumption-to-wealth ratio increases with respects to the volatility in Figure \ref{fig2} (a) and  Figure \ref{fig4} (a). Under Heston's stochastic volatility model, the market price of risks is $\xi\sqrt{\nu_{t}}$. If the volatility increases, then the market price of risks is higher, which implies that investment opportunities are improving. Then, the improvement in investment opportunities generates a positive income effect but a negative intertemporal substitution effect. Since the EIS parameter $\phi<1$, the income effect dominates the intertemporal substitution effect. Therefore, the investor has a higher consumption-to-wealth ratio when the volatility increases. Moreover, Figure \ref{fig2} (a) and Figure \ref{fig4} (a) illustrate that the optimal consumption-to-wealth ratio is higher when the risk aversion coefficient is smaller. This is because if the investor holds a lower risk aversion coefficient, then she/he will tend to increase the investment in the risky asset, which brings a more significant income effect and then leads to that the income effect is more dominant than the intertemporal substitution effect. These diverse phenomena demonstrate that it is interesting to consider consumption under the stochastic volatility model with recursive utility.

\begin{figure}[htbp]
	\begin{minipage}[t]{0.5\linewidth}
		\centering
		\includegraphics[width=3.0in]{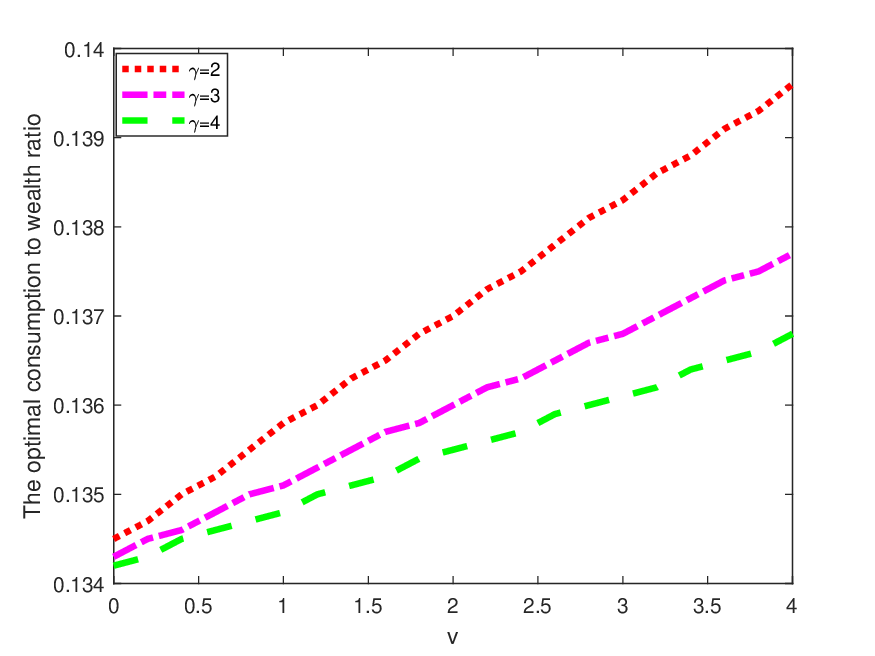}
        \caption*{(a)\;the effect of $\gamma$ on $\frac{\underline{c}^{*}_{0}}{X_{0}}$}
	\end{minipage}%
	\begin{minipage}[t]{0.5\linewidth}
		\centering
		\includegraphics[width=3.0in]{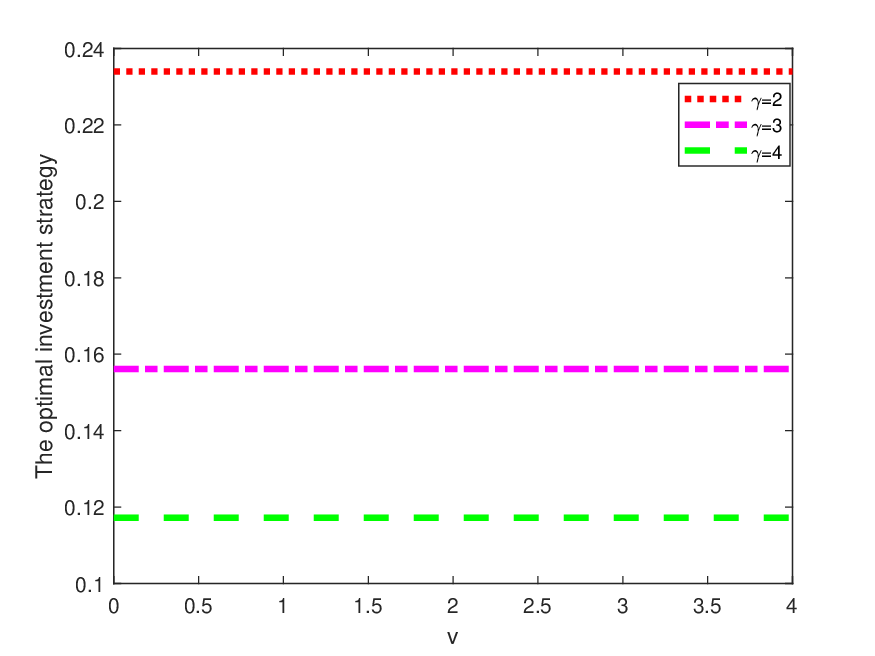}
        \caption*{(b)\;the effect of $\gamma$ on $\underline{\pi}^{*}_{0}$}
	\end{minipage}
    \caption{The effect of $\gamma$ on the optimal consumption-investment strategy \eqref{eq66} at time $t=0$.}
	\label{fig4}
\end{figure}

Figure \ref{fig2} (b) and Figure \ref{fig4} (b) disclose that the optimal investment strategy is insensitive to the volatility at initial time. It is worth mentioning that this phenomenon reflected in Figure \ref{fig4} (b) is consistent with Figure 4 in \cite{Kraft2013}. In addition, Figure \ref{fig3} (b) depicts that the optimal investment strategy is also insensitive to the time and only decreases very slightly in time. However, it can be clearly seen from (b) in Figures \ref{fig1}-\ref{fig4} that more risk-averse investor is more cautious by putting lower investment ratio in the risky asset, which is identical with our common sense.

\begin{figure}[htbp]
	\begin{minipage}[t]{0.5\linewidth}
		\centering
		\includegraphics[width=3.0in]{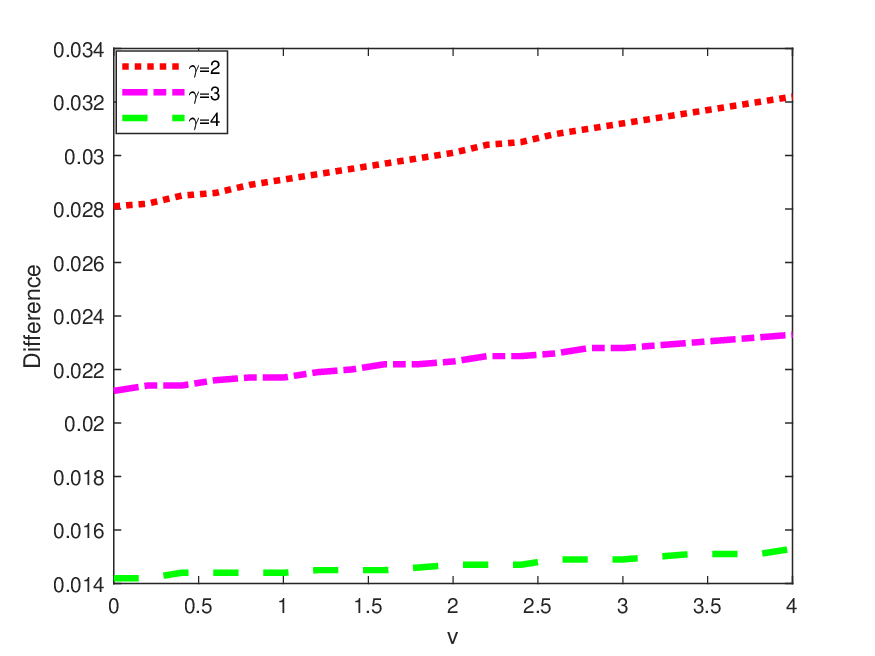}
        \caption*{(a)\; difference between $\frac{\underline{c}^{*}_{0}}{X_{0}}$ and $\frac{c^{\star}_{0}}{X^{\star}_{0}}$ in (5.2) in \cite{Kraft2013}}
	\end{minipage}%
	\begin{minipage}[t]{0.5\linewidth}
		\centering
		\includegraphics[width=3.0in]{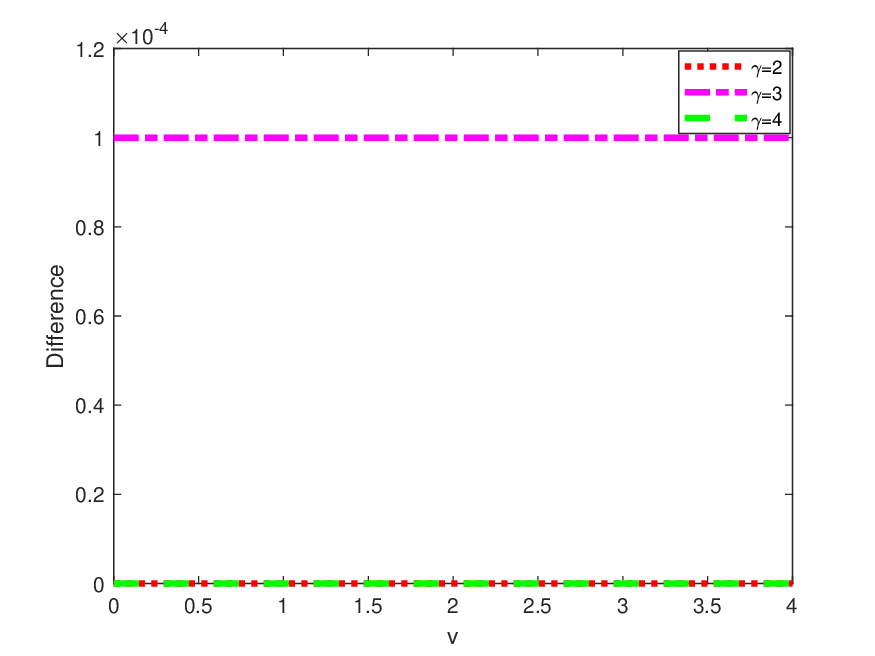}
        \caption*{(b)\; difference between $\underline{\pi}^{*}_{0}$ and $\pi^{\star}_{0}$ in (5.2) in \cite{Kraft2013}}
	\end{minipage}
    \caption{Difference between the optimal consumption-investment strategies \eqref{eq66} and (5.2) in \cite{Kraft2013} at time $t=0$.}
	\label{fig5}
\end{figure}

Showing the accuracy of the approximation solution \eqref{eq66}, Figure \ref{fig5} demonstrates the difference between the approximation solution \eqref{eq66} and the analytical solution (5.2) in \cite{Kraft2013} as a function of volatility at time $t=0$. We denote the difference between $\frac{\underline{c}^{*}_{0}}{X_{0}}$ and $\frac{c^{\star}_{0}}{X^{\star}_{0}}$ in (5.2) in \cite{Kraft2013} by $\frac{\underline{c}^{*}_{0}}{X_{0}}-\frac{c^{\star}_{0}}{X^{\star}_{0}}$. So, it follows from Figure \ref{fig5} (a) that the value of $\frac{\underline{c}^{*}_{0}}{X_{0}}$ is larger than that of $\frac{c^{\star}_{0}}{X^{\star}_{0}}$ in (5.2) in \cite{Kraft2013} and thus there exists a certain error. Moreover, the greater the volatility is, the greater the error is. However, when the value of the risk aversion coefficient $\gamma$ increases, the difference between $\frac{\underline{c}^{*}_{0}}{X_{0}}$ and $\frac{c^{\star}_{0}}{X^{\star}_{0}}$ in (5.2) in \cite{Kraft2013} becomes smaller even if the volatility value is high, which reflects higher accuracy of the approximation solution for consumption. On the other hand, we can see from Figure \ref{fig5} (b) that regardless of the values of the risk aversion coefficient and the volatility, the approximation solution for investment is extremely consistent with the analytical solution for investment, which again shows the higher accuracy of the approximation solution \eqref{eq66}.

\section{Conclusions}
This paper was devoted to investigating consumption and portfolio optimization problems under the general stochastic volatility model with recursive preferences in both infinite and finite time regions. In both cases of time regions, we further considered the investor with unit EIS and general EIS. By the dynamic programming approach, the optimization problems were changed into the solvability analysis of the corresponding HJB equations. Since the HJB equations usually do not have analytical solutions when the investor takes general EIS, we turned to find approximate solutions through the log-linear approximation method. By the conjecture of the exponential-polynomial form of the value function of the optimization problems under mild conditions, we proved that, when the order of the polynomial $n\leq2$, the HJB equation exists an analytical solution if the investor with unit EIS and an approximation solution by the log-linear approximation method otherwise. We also proved that the HJB equation has no analytical solutions or approximation solutions by the log-linear approximation method under the conjecture of the exponential-polynomial form of the value function when the order of the polynomial $n>2$. Moreover, some applications of our study to Heston's model were presented and expressions for the analytical and approximate solutions were also given. Finally, some numerical experiments were provided to illustrate the behavior of the optimal consumption-portfolio strategies.

Without further essential difficulties, our studies also work when considering multiple risky assets \cite{Shirzadi2020} and multiple state variable as well as ambiguity-aversion within the framework of \cite{Liu2010, Maenhout2004}. It is also interesting to derive the expressions for analytical or approximate solutions under more specific stochastic volatility models. We hope to present relevant study in the future.

\section*{Appendices}
\appendix
\renewcommand{\appendixname}{Appendix~\Alph{section}}

\section{Proof of Theorem \ref{theorem4.2}}
\begin{proof}
Similar to the proof of Theorem \ref{theorem4.1}, we first study the case when $n=3$ and then consider the case when $n>3$.

\noindent\textbf{Exponential-cubic:} Assume that $h(\nu)=\exp\{A_{0}+A_{1}\nu+\frac{1}{2}A_{2}\nu^{2}+\frac{1}{3}A_{3}\nu^{3}\}$.

Inserting the partial derivatives of $h$ and \eqref{eq28} into \eqref{eq26}, we can derive that
\begin{align}\label{eq30}
  r&-\zeta_{1}-\zeta_{2}\phi \ln\beta+\zeta_{2}(A_{0}+A_{1}\nu+\frac{1}{2}A_{2}\nu^{2}+\frac{1}{3}A_{3}\nu^{3})+\frac{1}{2\gamma}\eta^{2}(\nu)\nonumber\\
  &+\frac{\beta\phi}{\phi-1}\left[\frac{1}{\beta}\bigg(\zeta_{1}+\zeta_{2}(\phi \ln\beta-A_{0}-A_{1}\nu-\frac{1}{2}A_{2}\nu^{2}-\frac{1}{3}A_{3}\nu^{3})\bigg)-1\right]\nonumber\\
  &-\frac{1}{1-\phi}\left[m_{1}(\nu)+\frac{\eta(\nu)}{\gamma}\rho m_{2}(\nu)(1-\gamma)\right](A_{1}+A_{2} \nu+A_{3} \nu^{2})\nonumber\\
  &+\frac{ 1}{2(1-\phi)^{2}}m^{2}_{2}(\nu)
  \left[2-\phi-\gamma+\frac{\rho^{2}(1-\gamma)^{2}}{\gamma}\right](A_{1}+A_{2} \nu+A_{3} \nu^{2})^{2}\nonumber\\
  &-\frac{ 1}{2(1-\phi)}m^{2}_{2}(\nu)\left[(A_{1}+A_{2} \nu+A_{3} \nu^{2})^{2}+A_{2}+2A_{3} \nu\right]=0.
\end{align}
We can see that the above equation has terms involving $\nu^{4}$. If this equation can be solvable, one has to cancel the terms involving $\nu^{4}$. It implies that under the condition (B.7),
\begin{equation}\label{eq31}
  \frac{ 1}{(1-\phi)^{2}}\left[2-\phi-\gamma+\frac{\rho^{2}(1-\gamma)^{2}}{\gamma}\right]-\frac{ 1}{1-\phi}=0
\end{equation}
should be satisfied. It is easy to derive that the equation \eqref{eq31} is equivalent to $\rho^{2}=\frac{\gamma}{\gamma-1}$, which does not hold by the proof of Theorem \ref{theorem4.1}. Therefore, the terms involving $\nu^{4}$ in \eqref{eq30} cannot be matched. Thus, the conjecture of the exponential-cubic form of $h$ cannot approximately solve the problem \eqref{eq15} by the log-linear approximation method when $\phi\neq1$.

\noindent\textbf{Exponential-nth:} Assume that $h(\nu)=\exp\{A_{0}+\sum\limits_{k=1}^{n}\frac{1}{k}A_{k}\nu^{k}\}$ with $n>3$.

By substituting the partial derivatives of $h$ and \eqref{eq28} into \eqref{eq26}, we have
\begin{align*}
  r&-\zeta_{1}-\zeta_{2}\left(\phi \ln\beta-A_{0}-\sum\limits_{k=1}^{n}\frac{1}{k}A_{k}\nu^{k}\right)+\frac{\beta\phi}{\phi-1}\left[\frac{1}{\beta}\bigg(\zeta_{1}+\zeta_{2}(\phi \ln\beta-A_{0}-\sum\limits_{k=1}^{n}\frac{1}{k}A_{k}\nu^{k})\bigg)-1\right]\\
  &+\frac{1}{2\gamma}\eta^{2}(\nu)-\frac{1}{1-\phi}\left[m_{1}(\nu)+\frac{\eta(\nu)}{\gamma}\rho m_{2}(\nu)(1-\gamma)\right]\sum\limits_{k=1}^{n}A_{k}\nu^{k-1}\nonumber\\
  &+\frac{ 1}{2(1-\phi)^{2}}m^{2}_{2}(\nu)\left[2-\phi-\gamma+\frac{\rho^{2}(1-\gamma)^{2}}{\gamma}\right]\left(\sum\limits_{k=1}^{n}A_{k}\nu^{k-1}\right)^{2}\\
  &-\frac{ 1}{2(1-\phi)}m^{2}_{2}(\nu)\left[\left(\sum\limits_{k=1}^{n}A_{k}\nu^{k-1}\right)^{2}+\sum\limits_{k=2}^{n}(k-1)A_{k}\nu^{k-2}\right]=0,
\end{align*}
which involves $\nu^{n+1},\cdot\cdot\cdot,\nu^{2n-2}$ for $n>3$. Since the equation \eqref{eq31} does not hold, these terms cannot be matched under the condition (B.7). Therefore, given $n>3$, the conjecture of the exponential-polynomial form of $h$ cannot approximately solve the problem \eqref{eq15} by the log-linear approximation method when $\phi\neq1$.
\end{proof}

\section{Proof of Theorem \ref{theorem4.8}}
\begin{proof}
In the sequel, we will state out the solvability of approximate solutions of the PDE \eqref{eq26} in terms of $\eta(\nu)$, $m_{1}(\nu)$ and $m_{2}(\nu)$ under the log-linear approximation method.

\noindent\textbf{Exponential-constant:} Assume that $h(\nu)=\exp\{A_{0}\}$.

Substituting the partial derivatives of $h$ and \eqref{eq28} into \eqref{eq26} yields
\begin{equation*}
  r-\zeta_{1}-\zeta_{2}(\phi \ln\beta-A_{0})+\frac{1}{2\gamma}\eta^{2}(\nu)+\frac{\beta\phi}{\phi-1}\left[\frac{1}{\beta}\bigg(\zeta_{1}+\zeta_{2}(\phi \ln\beta-A_{0})\bigg)-1\right]=0,
\end{equation*}
which implies that the PDE \eqref{eq26} can be approximately solvable by the log-linear approximation method under the condition (B.8).

\noindent\textbf{Exponential-linear:} Assume that $h(\nu)=\exp\{A_{0}+A_{1}\nu\}$.

Substituting the partial derivatives of $h$ and \eqref{eq28} into \eqref{eq26} leads to
\begin{align*}
  r&-\zeta_{1}-\zeta_{2}(\phi \ln\beta-A_{0}-A_{1}\nu)+\frac{1}{2\gamma}\eta^{2}(\nu)+\frac{\beta\phi}{\phi-1}\left[\frac{1}{\beta}\bigg(\zeta_{1}+\zeta_{2}(\phi \ln\beta-A_{0}-A_{1}\nu)\bigg)-1\right]\\
  &-\frac{1}{1-\phi}\left[m_{1}(\nu)+\frac{\eta(\nu)}{\gamma}\rho m_{2}(\nu)(1-\gamma)\right]A_{1}\nonumber\\
  &+\frac{ 1}{2(1-\phi)^{2}}m^{2}_{2}(\nu)\left[2-\phi-\gamma+\frac{\rho^{2}(1-\gamma)^{2}}{\gamma}\right]A_{1}^{2}-\frac{ 1}{2(1-\phi)}m^{2}_{2}(\nu)A_{1}^{2}=0,
\end{align*}
which implies that the PDE \eqref{eq26} can be approximately solvable by the log-linear approximation method under the condition (B.9).

\noindent\textbf{Exponential-quadratic:} Assume that $h(\nu)=\exp\{A_{0}+A_{1}\nu+\frac{1}{2}A_{2}\nu^{2}\}$.

Similarly, one can obtain that
\begin{align*}
  r&-\zeta_{1}-\zeta_{2}(\phi \ln\beta-A_{0}-A_{1}\nu-\frac{1}{2}A_{2}\nu^{2})+\frac{\beta\phi}{\phi-1}\left[\frac{1}{\beta}\bigg(\zeta_{1}+\zeta_{2}(\phi \ln\beta-A_{0}-A_{1}\nu-\frac{1}{2}A_{2}\nu^{2})\bigg)-1\right]\\
  &+\frac{1}{2\gamma}\eta^{2}(\nu)-\frac{1}{1-\phi}\left[m_{1}(\nu)+\frac{\eta(\nu)}{\gamma}\rho m_{2}(\nu)(1-\gamma)\right](A_{1}+A_{2} \nu)\nonumber\\
  &+\frac{ 1}{2(1-\phi)^{2}}m^{2}_{2}(\nu)\left[2-\phi-\gamma+\frac{\rho^{2}(1-\gamma)^{2}}{\gamma}\right](A_{1}+A_{2} \nu)^{2}-\frac{ 1}{2(1-\phi)}m^{2}_{2}(\nu)\left[(A_{1}+A_{2} \nu)^{2}+A_{2}\right]=0,
\end{align*}
which means that the PDE \eqref{eq26} can be approximately solvable by the log-linear approximation method under the condition (B.7).
\end{proof}

\section{Proof of Theorem \ref{theorem4.3}}

\begin{proof}
In the following, we will show that the conjecture form \eqref{eq35} does not work in solving the PDE \eqref{eq34} under the cases when $n=3$ and $n>3$.

\noindent\textbf{Exponential-cubic:} Assume that $h(t,\nu)=\exp\{A_{0}(T-t)+A_{1}(T-t)\nu+\frac{1}{2}A_{2}(T-t)\nu^{2}+\frac{1}{3}A_{3}(T-t)\nu^{3}\}$.

By using the partial derivatives of $h$ and the equation \eqref{eq34}, we have
\begin{align}\label{eq36}
  -A'_{0}&-A'_{1}\nu-\frac{1}{2}A'_{2}\nu^{2}-\frac{1}{3}A'_{3}\nu^{3}+r-\beta+\frac{1}{2\gamma}\eta^{2}(\nu)
  +\beta(\ln\beta-A_{0}-A_{1}\nu-\frac{1}{2}A_{2}\nu^{2}-\frac{1}{3}A_{3}\nu^{3})\nonumber\\
  &+\left[m_{1}(\nu)+\frac{\eta(\nu)}{\gamma}\rho m_{2}(\nu)(1-\gamma)\right](A_{1}+A_{2}\nu+A_{3}\nu^{2})+\frac{ 1}{2}m^{2}_{2}(\nu)\nonumber\\
  &\times\left[(1-\gamma)(A_{1}+A_{2}\nu+A_{3}\nu^{2})^{2}+A_{2}+2A_{3}\nu\right]
  +\frac{\rho^{2}}{2\gamma}m^{2}_{2}(\nu)(1-\gamma)^{2}(A_{1}+A_{2}\nu+A_{3}\nu^{2})^{2}=0,
\end{align}
where $A'_{0}$ denotes the derivative of $A_{0}(t)$ with respect to $t$. Similar expressions will be used later when there is no ambiguity. Similar to the proof of Theorem \ref{theorem4.1}, we can know that the terms involving $\nu^{4}$ in \eqref{eq36} cannot be matched under the condition (B.7) and then the conjecture of the exponential-cubic form of $h$ cannot solve the PDE \eqref{eq34}.

\noindent\textbf{Exponential-nth:} Assume that $h(t,\nu)=\exp\{A_{0}(T-t)+\sum\limits_{k=1}^{n}\frac{1}{k}A_{k}(T-t)\nu^{k}\}$ with $n>3$.

By substituting the partial derivatives of $h$ into \eqref{eq34}, we can derive that
\begin{align*}
  -A'_{0}&-\sum\limits_{k=1}^{n}\frac{1}{k}A'_{k}\nu^{k}+r-\beta+\frac{1}{2\gamma}\eta^{2}(\nu)
  +\beta\left(\ln\beta-A_{0}-\sum\limits_{k=1}^{n}\frac{1}{k}A_{k}\nu^{k}\right)\\
  &  +\left[m_{1}(\nu)+\frac{\eta(\nu)}{\gamma}\rho m_{2}(\nu)(1-\gamma)\right]\sum\limits_{k=1}^{n}A_{k}\nu^{k-1}\nonumber\\
  &+\frac{ 1}{2}m^{2}_{2}(\nu)\left[(1-\gamma)\left(\sum\limits_{k=1}^{n}A_{k}\nu^{k-1}\right)^{2}+\sum\limits_{k=2}^{n}(k-1)A_{k}\nu^{k-2}\right]
  +\frac{\rho^{2}}{2\gamma}m^{2}_{2}(\nu)(1-\gamma)^{2}\left(\sum\limits_{k=1}^{n}A_{k}\nu^{k-1}\right)^{2}=0,
\end{align*}
which involves $\nu^{n+1},\cdot\cdot\cdot,\nu^{2n-2}$ for $n>3$. It follows from the proof of Theorem \ref{theorem4.1} that these terms cannot be matched under the condition (B.7). Thus, when $n>3$, the exponential-polynomial form of $h$ cannot solve the PDE \eqref{eq34}.
\end{proof}

\section{Proof of Theorem \ref{theorem4.10}}
\begin{proof}
We will investigate the solvability of the PDE \eqref{eq34} in terms of $\eta(\nu)$, $m_{1}(\nu)$ and $m_{2}(\nu)$ under the conjecture form \eqref{eq35} when $n=0,1,2$.

\noindent\textbf{Exponential-constant:} Assume that $h(t,\nu)=\exp\{A_{0}(T-t)\}$.

Substituting the partial derivatives of $h$ into \eqref{eq34} leads to
\begin{equation*}
  -A'_{0}+r-\beta+\frac{1}{2\gamma}\eta^{2}(\nu)+\beta(\ln\beta-A_{0})=0.
\end{equation*}
Therefore, the PDE \eqref{eq34} is solvable under the condition (B.8).

\noindent\textbf{Exponential-linear:} Assume that $h(t,\nu)=\exp\{A_{0}(T-t)+A_{1}(T-t)\nu\}$.

Inserting the partial derivatives of $h$ into \eqref{eq34} yields
\begin{align*}
   -A'_{0}&-A'_{1}\nu+r-\beta+\frac{1}{2\gamma}\eta^{2}(\nu)+\beta(\ln\beta-A_{0}-A_{1}\nu)+\left[m_{1}(\nu)+\frac{\eta(\nu)}{\gamma}\rho m_{2}(\nu)(1-\gamma)\right]A_{1}\\
  &+\frac{ 1}{2}m^{2}_{2}(\nu)\left[1-\gamma+\frac{\rho^{2}}{\gamma}(1-\gamma)^{2}\right]A_{1}^{2}=0,
\end{align*}
which shows that the PDE \eqref{eq34} is solvable under the condition (B.9).

\noindent\textbf{Exponential-quadratic:} Assume that $h(t,\nu)=\exp\{A_{0}(T-t)+A_{1}(T-t)\nu+\frac{1}{2}A_{2}(T-t)\nu^{2}\}$.

Similarly, one has
\begin{align*}
  -A'_{0}&-A'_{1}\nu-\frac{1}{2}A'_{2}\nu^{2}+r-\beta+\frac{1}{2\gamma}\eta^{2}(\nu)
  +\beta(\ln\beta-A_{0}-A_{1}\nu-\frac{1}{2}A_{2}\nu^{2})\\
  &+\left[m_{1}(\nu)+\frac{\eta(\nu)}{\gamma}\rho m_{2}(\nu)(1-\gamma)\right](A_{1}+A_{2}\nu)\\
  &+\frac{ 1}{2}m^{2}_{2}(\nu)\left[(1-\gamma)(A_{1}+A_{2}\nu)^{2}+A_{2}\right]
  +\frac{\rho^{2}}{2\gamma}m^{2}_{2}(\nu)(1-\gamma)^{2}(A_{1}+A_{2}\nu)^{2}=0,
\end{align*}
which means that the PDE \eqref{eq34} is solvable under the condition (B.7).
\end{proof}

\section{Proof of Theorem \ref{theorem4.4}}
\begin{proof}
We will verify that the conjecture form \eqref{eq41} does not work in solving the PDE \eqref{eq40} under the cases when $n=3$ and $n>3$, respectively.

\noindent\textbf{Exponential-cubic:} Assume that $h(t,\nu)=\exp\{A_{0}(T-t)+A_{1}(T-t)\nu+\frac{1}{2}A_{2}(T-t)\nu^{2}+\frac{1}{3}A_{3}(T-t)\nu^{3}\}$.

By substituting the partial derivatives of $h$ into \eqref{eq40}, we can derive that
\begin{align}\label{eq42}
  \frac{1}{1-\phi}&(A'_{0}+A'_{1}\nu+\frac{1}{2}A'_{2}\nu^{2}+\frac{1}{3}A'_{3}\nu^{3})+r-\zeta_{3}-\zeta_{4}\phi \ln\beta+\zeta_{4}(A_{0}+A_{1}\nu+\frac{1}{2}A_{2}\nu^{2}+\frac{1}{3}A_{3}\nu^{3})\nonumber\\
  &+\frac{1}{2\gamma}\eta^{2}(\nu)+\frac{\beta\phi}{\phi-1}\left[\frac{1}{\beta}\bigg(\zeta_{3}+\zeta_{4}(\phi \ln\beta-A_{0}-A_{1}\nu-\frac{1}{2}A_{2}\nu^{2}-\frac{1}{3}A_{3}\nu^{3})\bigg)-1\right]\nonumber\\
  &-\frac{1}{1-\phi}\left[m_{1}(\nu)+\frac{\eta(\nu)}{\gamma}\rho m_{2}(\nu)(1-\gamma)\right](A_{1}+A_{2} \nu+A_{3} \nu^{2})\nonumber\\
  &+\frac{ 1}{2(1-\phi)^{2}}m^{2}_{2}(\nu)
  \left[2-\phi-\gamma+\frac{\rho^{2}(1-\gamma)^{2}}{\gamma}\right](A_{1}+A_{2} \nu+A_{3} \nu^{2})^{2}\nonumber\\
  &-\frac{ 1}{2(1-\phi)}m^{2}_{2}(\nu)\left[(A_{1}+A_{2} \nu+A_{3} \nu^{2})^{2}+A_{2}+2A_{3} \nu\right]=0,
\end{align}
which has terms involving $\nu^{4}$. By the proof of Theorem \ref{theorem4.2}, the terms involving $\nu^{4}$ in \eqref{eq42} cannot be matched under the condition (B.7) and thus the exponential-cubic form of $h$ cannot solve the PDE \eqref{eq40}.

\noindent\textbf{Exponential-nth:} Assume that $h(t,\nu)=\exp\{A_{0}(T-t)+\sum\limits_{k=1}^{n}\frac{1}{k}A_{k}(T-t)\nu^{k}\}$ with $n>3$.

Substituting the partial derivatives of $h$ into \eqref{eq40}, we have
\begin{align*}
  \frac{1}{1-\phi}&(A'_{0}+\sum\limits_{k=1}^{n}\frac{1}{k}A'_{k}\nu^{k})+r-\zeta_{3}-\zeta_{4}\left(\phi \ln\beta-A_{0}-\sum\limits_{k=1}^{n}\frac{1}{k}A_{k}\nu^{k}\right)\\
  &+\frac{\beta\phi}{\phi-1}\left[\frac{1}{\beta}\bigg(\zeta_{3}+\zeta_{4}(\phi \ln\beta-A_{0}-\sum\limits_{k=1}^{n}\frac{1}{k}A_{k}\nu^{k})\bigg)-1\right]\\
  &+\frac{1}{2\gamma}\eta^{2}(\nu)-\frac{1}{1-\phi}\left[m_{1}(\nu)+\frac{\eta(\nu)}{\gamma}\rho m_{2}(\nu)(1-\gamma)\right]\sum\limits_{k=1}^{n}A_{k}\nu^{k-1}\nonumber\\
  &+\frac{ 1}{2(1-\phi)^{2}}m^{2}_{2}(\nu)\left[2-\phi-\gamma+\frac{\rho^{2}(1-\gamma)^{2}}{\gamma}\right]\left(\sum\limits_{k=1}^{n}A_{k}\nu^{k-1}\right)^{2}\\
  &-\frac{ 1}{2(1-\phi)}m^{2}_{2}(\nu)\left[\left(\sum\limits_{k=1}^{n}A_{k}\nu^{k-1}\right)^{2}+\sum\limits_{k=2}^{n}(k-1)A_{k}\nu^{k-2}\right]=0,
\end{align*}
which involves $\nu^{n+1},\cdot\cdot\cdot,\nu^{2n-2}$ for $n>3$. It follows from the proof of Theorem \ref{theorem4.2} that these terms cannot be matched under the condition (B.7). Thus, when $n>3$, the conjecture of the exponential-polynomial form of $h$ cannot solve the PDE \eqref{eq40}.
\end{proof}

\section{Proof of Theorem \ref{theorem4.12}}
\begin{proof}
We will point out the solvability of the PDE \eqref{eq40} in terms of $\eta(\nu)$, $m_{1}(\nu)$ and $m_{2}(\nu)$ under the conjecture form \eqref{eq41} when $n=0,1,2$, which reflects that the PDE \eqref{eq39} can be approximately solvable under the log-linear approximation method.

\noindent\textbf{Exponential-constant:} Assume that $h(t,\nu)=\exp\{A_{0}(T-t)\}$.

Substituting the partial derivatives of $h$ into \eqref{eq40} yields
\begin{equation*}
  \frac{A'_{0}}{1-\phi}+r-\zeta_{3}-\zeta_{4}(\phi \ln\beta-A_{0})+\frac{1}{2\gamma}\eta^{2}(\nu)+\frac{\beta\phi}{\phi-1}\left[\frac{1}{\beta}\bigg(\zeta_{3}+\zeta_{4}(\phi \ln\beta-A_{0})\bigg)-1\right]=0,
\end{equation*}
which implies that the PDE \eqref{eq40} can be solvable under the condition (B.8).

\noindent\textbf{Exponential-linear:} Assume that $h(t,\nu)=\exp\{A_{0}(T-t)+A_{1}(T-t)\nu\}$.

Similarly, one can obtain that
\begin{align*}
  \frac{1}{1-\phi}&(A'_{0}+A'_{1}\nu)+r-\zeta_{3}-\zeta_{4}(\phi \ln\beta-A_{0}-A_{1}\nu)+\frac{1}{2\gamma}\eta^{2}(\nu)\\
  &+\frac{\beta\phi}{\phi-1}\left[\frac{1}{\beta}\bigg(\zeta_{3}+\zeta_{4}(\phi \ln\beta-A_{0}-A_{1}\nu)\bigg)-1\right]
  -\frac{1}{1-\phi}\left[m_{1}(\nu)+\frac{\eta(\nu)}{\gamma}\rho m_{2}(\nu)(1-\gamma)\right]A_{1}\nonumber\\
  &+\frac{ 1}{2(1-\phi)^{2}}m^{2}_{2}(\nu)\left[2-\phi-\gamma+\frac{\rho^{2}(1-\gamma)^{2}}{\gamma}\right]A_{1}^{2}-\frac{ 1}{2(1-\phi)}m^{2}_{2}(\nu)A_{1}^{2}=0,
\end{align*}
which shows that the PDE \eqref{eq40} can be solvable under the condition (B.9).

\noindent\textbf{Exponential-quadratic:} Assume that $h(t,\nu)=\exp\{A_{0}(T-t)+A_{1}(T-t)\nu+\frac{1}{2}A_{2}(T-t)\nu^{2}\}$.

Inserting the partial derivatives of $h$ into \eqref{eq40} leads to
\begin{align*}
  \frac{1}{1-\phi}&(A'_{0}+A'_{1}\nu+\frac{1}{2}A'_{2}\nu^{2})+r-\zeta_{3}-\zeta_{4}(\phi \ln\beta-A_{0}-A_{1}\nu-\frac{1}{2}A_{2}\nu^{2})\\
  &+\frac{\beta\phi}{\phi-1}\left[\frac{1}{\beta}\bigg(\zeta_{3}+\zeta_{4}(\phi \ln\beta-A_{0}-A_{1}\nu-\frac{1}{2}A_{2}\nu^{2})\bigg)-1\right]\\
  &+\frac{1}{2\gamma}\eta^{2}(\nu)-\frac{1}{1-\phi}\left[m_{1}(\nu)+\frac{\eta(\nu)}{\gamma}\rho m_{2}(\nu)(1-\gamma)\right](A_{1}+A_{2} \nu)\nonumber\\
  &+\frac{ 1}{2(1-\phi)^{2}}m^{2}_{2}(\nu)\left[2-\phi-\gamma+\frac{\rho^{2}(1-\gamma)^{2}}{\gamma}\right](A_{1}+A_{2} \nu)^{2}-\frac{ 1}{2(1-\phi)}m^{2}_{2}(\nu)\left[(A_{1}+A_{2} \nu)^{2}+A_{2}\right]=0,
\end{align*}
which means that the PDE \eqref{eq40} can be solvable under the condition (B.7).
\end{proof}

\section{Proof of Proposition \ref{proposition5.1}}
\begin{proof}
{\rm(i)} Since $\eta^{2}(\nu)$, $m_{1}(\nu)$, $\eta(\nu) m_{2}(\nu)$ and $m^{2}_{2}(\nu)$ are linear in $\nu$, it follows from the case of exponential-linear in the proof of Theorem \ref{theorem4.6} that
\begin{align*}
  r&-\beta+\frac{1}{2\gamma}\xi^{2}\nu+\beta(\ln\beta-\hat{A}_{0}-\hat{A}_{1}\nu)+\left[\kappa(\theta-\nu)+\frac{\xi \sqrt{\nu}}{\gamma}\rho \sigma\sqrt{\nu}(1-\gamma)\right]\hat{A}_{1}\\
  &+\frac{ 1}{2}\sigma^{2}\nu\left[1-\gamma+\frac{\rho^{2}}{\gamma}(1-\gamma)^{2}\right]\hat{A}_{1}^{2}=0.
\end{align*}
Collecting constant terms and terms in $\nu$, we obtain $\hat{A}_{0}$ and $\hat{A}_{1}$ as solutions to the following equations
\begin{equation}\label{eq49}
  \hat{A}_{0}=\ln \beta-1+\frac{r+\kappa \theta \hat{A}_{1} }{\beta}
\end{equation}
and
\begin{equation}\label{eq50}
  \hat{a}_{1}(\hat{A}_{1})^{2}+\hat{a}_{2} \hat{A}_{1} + \hat{a}_{3}=0,
\end{equation}
where $\hat{a}_{1}=\frac{ \sigma^{2}}{2}\left[1-\gamma+\frac{\rho^{2}}{\gamma}(1-\gamma)^{2}\right]$, $\hat{a}_{2}=\frac{\xi \rho \sigma(1-\gamma)}{\gamma}-\beta-\kappa$ and $\hat{a}_{3}=\frac{\xi^{2}}{2\gamma}$. Since $\hat{a}_{1}<0$ and $\hat{a}_{3}>0$, the equation \eqref{eq50} that is quadratic in $\hat{A}_{1}$ has two opposite sign roots. We choose the positive root $\hat{A}_{1}>0$. This root ensures that $\hat{A}_{1}$ has a limit when $\gamma=1$, which is equivalent to the solution in the case of logarithmic utility. Given $\hat{A}_{1}$, one can know $\hat{A}_{0}$. Moreover, one can obtain \eqref{eq47} and \eqref{eq48}.

\textbf{The solution for} $\gamma=1$.

It is not hard to show that the HJB equation \eqref{eq18} with the logarithmic utility can be solved with $\hat{c}^{*}_{t}=\beta x$, $\hat{\psi}^{*}_{t}=\xi\sqrt{\nu}$ and $\hat{\omega}(x,\nu)=\ln x+\hat{A}_{0}+\hat{A}_{1}\nu$, where $\hat{A}_{0}$ and $\hat{A}_{1}$ are given by equations \eqref{eq49} and \eqref{eq50}, respectively. The coefficients in the equation \eqref{eq50} satisfy $\hat{a}_{1}=0$, $\hat{a}_{2}=-\beta-\kappa$ and $\hat{a}_{3}=\frac{\xi^{2}}{2}$. There exists only one root
\begin{equation}\label{eq51}
  \hat{A}_{1}=\frac{\xi^{2}}{2(\beta+\kappa)}>0.
\end{equation}
Thus, only the positive root in \eqref{eq50} has a limit equals \eqref{eq51} when $\gamma=1$.

{\rm(ii)} Similarly, from the case of exponential-linear in the proof of Theorem \ref{theorem4.8}, we can derive that
\begin{align*}
  r&-\check{\zeta}_{1}-\check{\zeta}_{2}(\phi \ln\beta-\check{A}_{0}-\check{A}_{1}\nu)+\frac{1}{2\gamma}\xi^{2}\nu+\frac{\beta\phi}{\phi-1}\left[\frac{1}{\beta}\bigg(\check{\zeta}_{1}+\check{\zeta}_{2}(\phi \ln\beta-\check{A}_{0}-\check{A}_{1}\nu)\bigg)-1\right]\\
  &-\frac{1}{1-\phi}\left[\kappa(\theta-\nu)+\frac{\xi \sqrt{\nu}}{\gamma}\rho \sigma\sqrt{\nu}(1-\gamma)\right]\check{A}_{1}\nonumber\\
  &+\frac{ \sigma^{2}\nu}{2(1-\phi)^{2}}\left[2-\phi-\gamma+\frac{\rho^{2}(1-\gamma)^{2}}{\gamma}\right]\check{A}_{1}^{2}-\frac{ \sigma^{2}\nu}{2(1-\phi)}\check{A}_{1}^{2}=0,
\end{align*}
where $\check{\zeta}_{1}=\exp\{\mathbb{E}(\ln\check{c}^{*}_{t}-\ln x)\}\left[1-\mathbb{E}(\ln\check{c}^{*}_{t}-\ln x)\right]$ and $\check{\zeta}_{2}=\exp\{\mathbb{E}(\ln\check{c}^{*}_{t}-\ln x)\}$ with the optimal consumption strategy $\check{c}^{*}_{t}$ and the corresponding wealth value $x$. Comparing constant terms and terms in $v$, we know that $\check{A}_{0}$ and $\check{A}_{1}$ satisfy
\begin{equation}\label{eq52}
  \check{A}_{0}=\phi\ln \beta+\frac{r(\phi-1)+\check{\zeta}_{1}-\beta\phi+\kappa\theta\check{A}_{1}}{\check{\zeta}_{2}}
\end{equation}
and
\begin{equation}\label{eq53}
  \check{a}_{1}(\check{A}_{1})^{2}+\check{a}_{2} \check{A}_{1} + \check{a}_{3}=0,
\end{equation}
where $\check{a}_{1}=\frac{ \sigma^{2}}{2(1-\phi)^{2}}\left[1-\gamma+\frac{\rho^{2}(1-\gamma)^{2}}{\gamma}\right]$, $\check{a}_{2}=\frac{1}{1-\phi}\left[\check{\zeta}_{2}+\kappa-\frac{\xi \rho \sigma(1-\gamma)}{\gamma}\right]$ and $\check{a}_{3}=\frac{\xi^{2}}{2\gamma}$. Because $\check{a}_{1}<0$ and $\check{a}_{3}>0$, the equation \eqref{eq53} has two opposite sign roots. We choose the root such that $-\frac{\check{A}_{1}}{1-\phi}>0$ to guarantee the approximate solution can approach to the solution when $\phi=1$. Given $\check{A}_{1}$, we know $\check{A}_{0}$. In addition, we can obtain \eqref{eq54} and \eqref{eq55}.

{\rm(iii)} By the case of exponential-linear in the proof of Theorem \ref{theorem4.10}, one has
\begin{align*}
   -\tilde{A}'_{0}&-\tilde{A}'_{1}\nu+r-\beta+\frac{1}{2\gamma}\xi^{2}\nu+\beta(\ln\beta-\tilde{A}_{0}-\tilde{A}_{1}\nu)+\left[\kappa(\theta-\nu)+\frac{\xi \sqrt{\nu}}{\gamma}\rho \sigma\sqrt{\nu}(1-\gamma)\right]\tilde{A}_{1}\\
  &+\frac{\sigma^{2}\nu}{2}\left[1-\gamma+\frac{\rho^{2}}{\gamma}(1-\gamma)^{2}\right]\tilde{A}_{1}^{2}=0
\end{align*}
with boundary conditions $\tilde{A}_{0}(0)=\ln \varepsilon$ and $\tilde{A}_{1}(0)=0$. Collecting constant terms and terms in $v$, one can derive that
\begin{equation}\label{eq60}
  \tilde{A}'_{0}+\beta\tilde{A}_{0}=r-\beta+\beta \ln\beta+\kappa\theta\tilde{A}_{1}
\end{equation}
and
\begin{equation}\label{eq61}
 \tilde{A}'_{1}=\frac{\tilde{a}_{1}}{2}\tilde{A}_{1}^{2}-\tilde{a}_{2}\tilde{A}_{1}+\frac{\tilde{a}_{3}}{2},
\end{equation}
where $\tilde{a}_{1}=\sigma^{2}\left[1-\gamma+\frac{\rho^{2}}{\gamma}(1-\gamma)^{2}\right]$, $\tilde{a}_{2}=\beta+\kappa-\frac{\xi \rho \sigma(1-\gamma)}{\gamma}$ and $\tilde{a}_{3}=\frac{\xi^{2}}{\gamma}$. Let $\tilde{a}_{4}=\sqrt{\tilde{a}^{2}_{2}-\tilde{a}_{1}\tilde{a}_{3}}$. For \eqref{eq61}, we have
\begin{equation}\label{eq57}
 \tilde{A}_{1}(\tau)=\frac{\tilde{a}_{3}(e^{\tilde{a}_{4}\tau}-1)}{2\tilde{a}_{4}+(\tilde{a}_{2}+\tilde{a}_{4})(e^{\tilde{a}_{4}\tau}-1)}
\end{equation}
with $\tau=T-t$. Then we can obtain
\begin{equation}\label{eq56}
  \tilde{A}_{0}(\tau)=\ln \varepsilon \times e^{-\beta \tau}+\left(\frac{r}{\beta}-1+\ln\beta\right)(1-e^{-\beta \tau})+\kappa\theta\int_{0}^{\tau}\tilde{A}_{1}(s)e^{-\beta (\tau-s)}\mathrm{d}s.
\end{equation}
Moreover, we have \eqref{eq58} and \eqref{eq59}.

{\rm(iv)} According to the case of exponential-linear in the proof of Theorem \ref{theorem4.12}, we can obtain
\begin{align*}
  \frac{1}{1-\phi}&(\underline{A}'_{0}+\underline{A}'_{1}\nu)+r-\underline{\zeta}_{3}-\underline{\zeta}_{4}(\phi \ln\beta-\underline{A}_{0}-\underline{A}_{1}\nu)+\frac{1}{2\gamma}\xi^{2}\nu\\
  &+\frac{\beta\phi}{\phi-1}\left[\frac{1}{\beta}\bigg(\underline{\zeta}_{3}+\underline{\zeta}_{4}(\phi \ln\beta-\underline{A}_{0}-\underline{A}_{1}\nu)\bigg)-1\right]
  -\frac{1}{1-\phi}\left[\kappa(\theta-\nu)+\frac{\xi \sqrt{\nu}}{\gamma}\rho \sigma\sqrt{\nu}(1-\gamma)\right]\underline{A}_{1}\nonumber\\
  &+\frac{ \sigma^{2}\nu}{2(1-\phi)^{2}}\left[2-\phi-\gamma+\frac{\rho^{2}(1-\gamma)^{2}}{\gamma}\right]\underline{A}_{1}^{2}-\frac{ \sigma^{2}\nu}{2(1-\phi)}\underline{A}_{1}^{2}=0,
\end{align*}
where $\underline{\zeta}_{3}=\exp\{(\ln\underline{c}^{*}_{t}-\ln x)|_{\nu_{t}=\theta}\}\left[1-(\ln\underline{c}^{*}_{t}-\ln x)|_{\nu_{t}=\theta}\right]$ and $\underline{\zeta}_{4}=\exp\{(\ln\underline{c}^{*}_{t}-\ln x)|_{\nu_{t}=\theta}\}$ with the optimal consumption strategy $\underline{c}^{*}_{t}$ and the corresponding wealth value $x$. Comparing constant terms and terms in $v$ leads to
\begin{equation}\label{eq62}
  \underline{A}'_{0}+\underline{\zeta}_{4}\underline{A}_{0}=r(\phi-1)+\underline{\zeta}_{3}+\underline{\zeta}_{4}\phi \ln\beta-\beta\phi+\kappa\theta\underline{A}_{1}
\end{equation}
and
\begin{equation}\label{eq63}
 \underline{A}'_{1}=\frac{\underline{a}_{1}}{2}\underline{A}_{1}^{2}-\underline{a}_{2}\underline{A}_{1}+\frac{\underline{a}_{3}}{2},
\end{equation}
where $\underline{a}_{1}=\frac{\sigma^{2}}{\phi-1}\left[1-\gamma+\frac{\rho^{2}}{\gamma}(1-\gamma)^{2}\right]$, $\underline{a}_{2}=\underline{\zeta}_{4}+\kappa-\frac{\xi \rho \sigma(1-\gamma)}{\gamma}$ and $\underline{a}_{3}=\frac{\xi^{2}(\phi-1)}{\gamma}$.
Thus, for \eqref{eq62}, one has
\begin{equation}\label{eq65}
  \underline{A}_{0}(\tau)=\ln \varepsilon \times e^{-\int_{0}^{\tau}\underline{\zeta}_{4}(s)\mathrm{d}s}+\int_{0}^{\tau}\left[r(\phi-1)+\underline{\zeta}_{3}(s)+\underline{\zeta}_{4}(s)\phi \ln\beta-\beta\phi+\kappa\theta\underline{A}_{1}(s)\right]e^{- \int_{s}^{\tau}\underline{\zeta}_{4}(u)\mathrm{d}u}\mathrm{d}s.
\end{equation}
In addition, we obtain \eqref{eq66} and \eqref{eq67}.

Overall, it follows from the cases of (i)-(iv) that the optimal consumption strategies are proportional to the current wealth value $X_{t}$ and the optimal investment strategies are independent of $X_{t}$. Thus, substituting the optimal consumption-investment strategy under the cases of (i)-(iv) into \eqref{eq46} respectively, we can know that $X_{t}$ satisfies geometric Brownian motion. Given a positive initial value $X_{0}$, $X_{t}$ remains positive for $t>0$, which implies that the optimal consumption strategies for the cases of (i)-(iv) are also positive. Moreover, for the cases (i) and (iii) , it can be verified whether the consumption strategies $\hat{c}^{*}_{t}$ and $\tilde{c}^{*}_{t}$ are within the interval $[C_{m},C_{M}]$ based on the specific values of the model parameters.

Furthermore, for the case (i) with the optimal consumption-investment strategy \eqref{eq47} and the corresponding wealth process $\hat{X}_{t}$, one has
\begin{align*}
  \mathbb{E}_{t}\left[\int^{\infty}_{t}\left|f(\hat{c}^{*}_{s},J_{s})\right|\mathrm{d}s\right]&=
  \mathbb{E}_{t}\left[\left|\frac{\hat{X}_{t}^{1-\gamma}}{1-\gamma}e^{(1-\gamma)(\hat{A}_{0}+\hat{A}_{1}\nu_{t})}\right|\right]\nonumber\\
  &=\mathbb{E}_{t}\left[\left|\frac{x_{0}^{1-\gamma}}{1-\gamma}e^{(1-\gamma)\left[\int^{t}_{0}\left(r-\beta+\xi\hat{\pi}^{*}_{s}\nu_{s}-\frac{1}{2}\hat{\pi}_{s}^{2}\nu_{s}\right)\mathrm{d}s
  +\int^{t}_{0}\hat{\pi}^{*}_{s}\sqrt{\nu_{s}}(\rho\mathrm{d}W^{\nu}_{s}+\sqrt{1-\rho^{2}}\mathrm{d}W^{\perp}_{s})+\hat{A}_{0}+\hat{A}_{1}\nu_{t}\right]}\right|\right]\nonumber\\
  &\leq K\mathbb{E}_{t}\left[\left|e^{(1-\gamma)\left[\int^{t}_{0}\left(\xi\hat{\pi}^{*}_{s}\nu_{s}-\frac{1}{2}\hat{\pi}_{s}^{2}\nu_{s}\right)\mathrm{d}s
  +\int^{t}_{0}\hat{\pi}^{*}_{s}\sqrt{\nu_{s}}(\rho\mathrm{d}W^{\nu}_{s}+\sqrt{1-\rho^{2}}\mathrm{d}W^{\perp}_{s})\right]}\right|\right],
\end{align*}
where the last inequality follows from the fact that $\gamma>1$, $\hat{A}_{1}>0$ and $\nu_{t}>0$. Next, we aim to find an estimate for $e^{(1-\gamma)\left[\int^{t}_{0}\left(\xi\hat{\pi}^{*}_{s}\nu_{s}-\frac{1}{2}\hat{\pi}_{s}^{2}\nu_{s}\right)\mathrm{d}s
  +\int^{t}_{0}\hat{\pi}^{*}_{s}\sqrt{\nu_{s}}(\rho\mathrm{d}W^{\nu}_{s}+\sqrt{1-\rho^{2}}\mathrm{d}W^{\perp}_{s})\right]}$. We note that
\begin{align*}
e^{(1-\gamma)\left[\int^{t}_{0}\left(\xi\hat{\pi}^{*}_{s}\nu_{s}-\frac{1}{2}\hat{\pi}_{s}^{2}\nu_{s}\right)\mathrm{d}s
  +\int^{t}_{0}\hat{\pi}^{*}_{s}\sqrt{\nu_{s}}(\rho\mathrm{d}W^{\nu}_{s}+\sqrt{1-\rho^{2}}\mathrm{d}W^{\perp}_{s})\right]}
=&\underbrace{e^{(1-\gamma)\int^{t}_{0}\left(\xi\hat{\pi}^{*}_{s}\nu_{s}-\frac{1}{2}\hat{\pi}_{s}^{2}\nu_{s}\right)\mathrm{d}s
+(1-\gamma)^{2}\int^{t}_{0}\hat{\pi}_{s}^{2}\nu_{s}\mathrm{d}s}}_{A}\\
&\times\underbrace{e^{-(1-\gamma)^{2}\int^{t}_{0}\hat{\pi}_{s}^{2}\nu_{s}\mathrm{d}s
  +(1-\gamma)\int^{t}_{0}\hat{\pi}^{*}_{s}\sqrt{\nu_{s}}(\rho\mathrm{d}W^{\nu}_{s}+\sqrt{1-\rho^{2}}\mathrm{d}W^{\perp}_{s})}}_{B}.
\end{align*}
For the term B, since $(1-\gamma)\hat{\pi}^{*}_{s}$ is  deterministic and bounded on $[0,\infty)$, it follows from Lemma 4.3 in \cite{ZengX2013} that
\begin{equation}\label{eq88}
  \mathbb{E}[B^{2}]=\mathbb{E}\left[e^{-2(1-\gamma)^{2}\int^{t}_{0}\hat{\pi}_{s}^{2}\nu_{s}\mathrm{d}s
  +2(1-\gamma)\int^{t}_{0}\hat{\pi}^{*}_{s}\sqrt{\nu_{s}}(\rho\mathrm{d}W^{\nu}_{s}+\sqrt{1-\rho^{2}}\mathrm{d}W^{\perp}_{s})}\right]<\infty.
\end{equation}
For the term A, we can derive that
$$
\mathbb{E}[A^{2}]=e^{2(1-\gamma)\int^{t}_{0}\left(\xi\hat{\pi}^{*}_{s}\nu_{s}-\frac{1}{2}\hat{\pi}_{s}^{2}\nu_{s}\right)\mathrm{d}s
+2(1-\gamma)^{2}\int^{t}_{0}\hat{\pi}_{s}^{2}\nu_{s}\mathrm{d}s.}
$$
By Theorem 5.1 in \cite{ZengX2013}, one has that $\mathbb{E}[A^{2}]<\infty$ if the assumption
$$
2(1-\gamma)\xi\hat{\pi}^{*}_{s}-(1-\gamma)\hat{\pi}_{s}^{2}+2(1-\gamma)^{2}\hat{\pi}_{s}^{2}\leq\frac{\kappa^{2}}{2\sigma^{2}}
$$
is satisfied, which is equivalent to the inequality
$$
(1-2\gamma)\rho^{2}\sigma^{2}(1-\gamma)^{2}\hat{A}_{1}^{2}+2\xi\rho\sigma(1-\gamma)^{2}\hat{A}_{1}+\xi^{2}\leq\frac{\kappa^{2}\gamma^{2}}{2(1-\gamma)\sigma^{2}}.
$$
Thus, one has
$$
\mathbb{E}_{t}\left[\int^{\infty}_{t}\left|f(\hat{c}^{*}_{s},J_{s})\right|\mathrm{d}s\right]\leq K\mathbb{E}_{t}\left[AB\right]\leq K \left(\mathbb{E}\left[A^{2}\right] \mathbb{E}\left[B^{2}\right]\right)^{\frac{1}{2}}<\infty.
$$
The arguments for the cases (ii)-(iv) are similar to those in the case (i), we skip a detailed discussion.
\end{proof}

\end{document}